\documentclass[bj,numbers,preprint]{imsart}

\usepackage{graphicx}
\usepackage{natbib}
\usepackage{amsmath,amsfonts,amsthm, amssymb,latexsym, mathtools}
\usepackage{xcolor}
\usepackage{hyperref}
\usepackage{bbm}
\usepackage{mathrsfs}
\usepackage{booktabs}

\startlocaldefs

\theoremstyle{plain}
\newtheorem{lemma}{Lemma}

\newtheorem{theorem}{Theorem}

\theoremstyle{definition}
\newtheorem{remark}{Remark}

\newcommand{\post}[1]{\Pi \left[ #1 \, | \, X \right] }
\newcommand{\R}{\mathbb{R}}

\definecolor{commentpaul}{rgb}{0,0.6,0}
\definecolor{commentser}{rgb}{0.6,0,0}

\definecolor{amethyst}{rgb}{0.6, 0.4, 0.8}

\definecolor{blendedblue}{rgb}{0.2,0.2,0.7}

\makeatletter
\newcommand{\mylabel}[2]{#2\def\@currentlabel{#2}\label{#1}}
\makeatother

\endlocaldefs

\begin{document}

\begin{frontmatter}
\title{Heavy-tailed and Horseshoe priors for regression and
sparse Besov rates}
\runtitle{Heavy-tailed priors and sparse Besov rates}

\begin{aug}
\author[A]{\fnms{Sergios} \snm{Agapiou}\ead[label=e1]{agapiou.sergios@ucy.ac.cy}}
\author[B]{\fnms{Isma\"el}~\snm{Castillo}\ead[label=e2]{ismael.castillo@sorbonne-universite.fr}}
\author[C]{\fnms{Paul}~\snm{Egels}\ead[label=e3]{paul.egels@sorbonne-universite.fr}}
\address[A]{Department of Mathematics and Statistics,
University of Cyprus, Nicosia, Cyprus.
\printead{e1}}

\address[B]{Sorbonne Université, LPSM; 4, place Jussieu, 75005 Paris,
France.
\printead{e2}}

\address[C]{Sorbonne Université, LPSM; 4, place Jussieu, 75005 Paris,
France.
\printead{e3}}
\end{aug}

\begin{abstract}
The large variety of functions encountered in nonparametric statistics, calls for methods that are flexible enough to achieve optimal or near-optimal performance over a wide variety of functional classes, such as Besov balls, as well as over a large array of loss functions. In this work, we show that a class of heavy-tailed prior distributions on basis function coefficients introduced in \cite{AC} and called Oversmoothed heavy-Tailed (OT) priors, leads to Bayesian posterior distributions that satisfy these requirements;  the case of horseshoe distributions is also investigated, for the first time in the context of nonparametrics, and we show that they fit into this framework. 
Posterior contraction rates are derived in two settings. The case of Sobolev--smooth signals and $L_2$--risk is considered first, along with a lower bound result showing that the imposed form of the scalings on prior coefficients by the OT prior is necessary to get full adaptation to smoothness. Second, the broader case of Besov-smooth signals with $L_{p'}$--risks, $p' \geq 1$, is considered, and minimax posterior contraction rates, adaptive to the underlying smoothness, and including rates in the so-called {\em sparse} zone, are derived.   
We  provide an implementation of the proposed method and illustrate our results through a simulation study.
\end{abstract}


\end{frontmatter}

\section{Introduction}\label{sec : intro}

The ability of estimators to be flexible, in the sense that their properties remain excellent in a variety of contexts, is a particularly sought-after characteristic. In nonparametric statistics, where the quantity of interest is typically an unknown function (which can represent a signal, an image etc. in practical applications), wavelet thresholding methods are a fundamental class that arguably satisfies the previous flexibility desideratum. Simple non-linear thresholding rules such as hard or soft thresholding are very broadly used in many areas of statistics and signal processing; from the mathematical perspective, they achieve asymptotic near-minimaxity over a very broad variety of functional classes and loss functions. While they are simple to implement and overall display very good empirical behaviour, wavelet thresholding methods still face a number of challenges numerically, in particular for low or moderate sample sizes, as one needs to find good values for thresholding constants, which may require some tuning.

Bayesian nonparametric methods, on the other hand, have experienced rapid growth since the early 2000's, when the development of sampling algorithms to simulate from posterior distributions or approximations thereof (MCMC, ABC, Variational Bayes to name just a few) has been paralleled by a progressive mathematical understanding of the properties needed for the prior distribution to achieve optimal or near-optimal convergence properties for the corresponding posterior distributions. We refer to \cite{gvbook, icsfl} for  recent overviews on the field. Among the most popular classes of priors on functions in statistics and machine learning are Gaussian processes \cite{rw06} (henceforth GPs). For a number of function classes, in particular ones for which the signal's smoothness is `homogeneous' over the considered input space, one can show that GPs achieve optimal posterior contraction rates, provided their parameters are well chosen \cite{vz08, ic08}, and adaptation to smoothness can be achieved by, for example, drawing a scaling parameter at random \cite{vz09}. However, for more `inhomogeneous' signals, Gaussian processes can be shown to lead to suboptimal rates \cite{aw21} for certain losses at least (e.g. quadratic) and, more generally, properties of posterior distributions for functions with `heterogeneous' regularity across the input space are currently less understood; we review below a few recent  results in this direction.  
   
Regarding optimality of rates, historically the first systematic study of non-matched situations where the parameter of the loss function does not match the norm index defining the functional class, is due to Nemirovskii \cite{nem85} and to Nemirovskii, Polyak and Tsybakov \cite{npt83, npt85}, where in particular it was noticed that linear estimators can be suboptimal. Donoho and Johnstone \cite{DJ94} studied this type of questions in  sequence space, while 
Donoho, Johnstone, Kerkyacharian and Picard \cite{djkp96} investigated  {\em adaptive} minimax rates in density estimation for Besov classes. For such classes, it turns out that the information-theoretic boundaries can be described by three `zones', depending on the loss function and the characteristics of the Besov space -- we refer to \cite{hkpt98} for an in-depth discussion --: a {\em regular} zone, where best linear estimators achieve the minimax rate and the rate is the classical $n^{-\beta/(2\beta+1)}$ rate in terms of regularity $\beta$ and number of observations $n$, an {\em intermediate} zone, where the `global' minimax rate is still the same but linear estimators are suboptimal, and a {\em sparse} zone, where the rate deteriorates and can be expressed as a combination of the regularity and loss function parameters. This illustrates the delicate interplay between measures of loss and regularity, which in particular leads to significant, polynomial--in--$n$ differences in convergence rates. Aside of wavelet methods, an adaptive local bandwidth selector based on Lepski's method was developed for kernel--based estimators in \cite{lepmamspok97} and latter for anisotropic functions in \cite{Lepski}. While working on this paper, we learned of the very recent work \cite{lmr25}, where the authors derive oracle inequalities and convergence rates in non-matched situations for penalisation methods. 
 
Coming back to Bayesian methods, as noted above, Gaussian Processes can be suboptimal when estimating functions with inhomogeneous smoothness (a fact formally proved in \cite{aw21} in the white noise model, with mean square loss).  
Recently, it has been noted that taking heavier tails than Gaussian in defining priors can help in obtaining more diverse behaviour. In \cite{adh21}, $p$-exponential priors on coefficients of a basis expansion are considered; while the posterior convergence rates in mean square loss improve in a range of aspects compared to the Gaussian case, they are still not adaptive to smoothness. A main advantage is that a wider range of Besov spaces is accessible with optimal contraction rates with these priors, for well chosen parameters. For example, unlike GP priors, priors with exponential tails can be tuned to attain the minimax rate in mean square loss, over both classes of homogeneously smooth functions as well as classes permitting spatial inhomogeneities, see \cite{adh21} and the recent preprint \cite{doleraetal24}. To make these priors adaptive, it is still necessary, however, and similar to GPs as discussed above, to draw at least one extra parameter at random;  such adaptive counterparts have  been obtained in \cite{as22} for the white noise model, while \cite{g23} considered density estimation. 
In \cite{AC}, the following (surprising at first) result is obtained: putting heavy-tailed priors (in the sense that they have polynomially decreasing tails) on coefficients can lead to {\em automatic} adaptation to smoothness. By taking well-chosen (but, crucially, deterministic and universal) scalings on heavy-tailed coefficients -- defining so-called Oversmoothed heavy-Tailed (OT) priors --, the authors in  \cite{AC} derive adaptation, up to a logarithmic factor, to homogeneous and inhomogeneous smoothness in the $L_2$--norm, and in the $L_\infty$ (supremum)--norm over H\"older classes. Results cover Gaussian white noise regression using standard posterior distributions and priors with at least two moments (excluding e.g. Cauchy or horseshoe priors), while extensions are provided to other models such as density estimation and binary classification using fractional posteriors.  
The work \cite{ce24} uses related heavy-tailed priors on coefficients of deep neural networks to derive automatic adaptation to smoothness and compositional structures or to smoothness and geometric structures, as well as adaptation to some anisotropic Besov classes, but results therein are confined to the $L_2$--loss.

We now review a few of the existing results for Bayesian posteriors for `non-canonical distances' (in the loose sense of distances for which one cannot directly apply a generic contraction rate theorem as in \cite{ggv00}).  
For the Bayesian approach, which is likelihood-based, the study of convergence in terms of certain loss functions, such as $L_p$ ($p>2$) or $L_\infty$--losses, is notoriously difficult. The first systematic study of posterior contraction in such norms is due to Giné and Nickl \cite{gn11}, where consistency (in general sub-minimax) rates are obtained under generic conditions. An approach to get posterior contraction at minimax rate in the supremum norm is introduced in \cite{ic14}. Adaptive rates have been obtained for specific priors and models \cite{hrs15, cm21, cr22, naulet22, AC}.
The work  \cite{cr21} shows contraction under the (global) supremum norm  for tree-type BCART priors, while \cite{rr24} proves local rates for spike-and-slab priors and  BCART priors. 
Yet, to the best of our knowledge, there is no systematic study of Bayesian posterior distributions for Besov regularities and loss functions other than $L_2$ and $L_\infty$ and one of the paper's aims is to fill this gap in a regression setting. To reduce technicalities, we will do so in the Gaussian white noise model, that can be seen as a prototypical nonparametric model \cite{GineNickl}.

We also underline that, while `adaptive' priors (in the sense that, for example, they adapt to the unknown underlying smoothness) can sometimes be obtained by drawing certain parameters in the prior at random, in general -- especially for distances for which one cannot apply a generic result as in \cite{ggv00} -- it is
 unclear how to get adaptation to smoothness in terms of a specific loss, such as an $L_p$--loss. To exemplify this, let us consider the class of sieve priors with random truncation, namely priors defined as functions expanded on a finite number $K$ of coefficients over a basis, with random coefficients, and where the cut-off $K$ is itself drawn at random. The work \cite{arbel13}, for instance, shows that such priors are typically adaptive in terms of the $L_2$--loss in regression, while a certain sub-optimality is noted for the posterior mean in pointwise loss. In their study of Bayesian trees, the authors in \cite{cr21}, prove a lower bound result (Theorem 5 in that source) showing that the posterior contraction rate of a fully-grown tree with random depth $K$ (which is a regular histogram prior and can be viewed as a sieve with cut-off $K$), while optimal for the $L_2$-loss (up to a logarithmic term), is {\em polynomially} sub-optimal in terms of a supremum-type loss. This illustrates that the choice of the prior to target adaptation with respect to a certain loss must be done with particular care. 
 
Regarding specific priors, it has been reported in the literature that {\em horseshoe} priors \cite{hs1, hs2}, that are a very popular choice in the context of sparse high-dimensional models (and can be shown to display near-optimal behaviour in these sparse settings \cite{vkv14, vsv17}), can also be deployed in nonparametrics with empirically very good behaviour, but up to now there is no theory backing up these empirical findings.

The contributions of this paper are threefold
\begin{enumerate}
\item we consider the question of obtaining minimax convergence rates for posterior distributions over a wide variety of Besov classes and losses. We
show that the oversmoothed heavy-tailed (OT) priors recently introduced in \cite{AC} (and featuring scaling parameters with a specific deterministic decrease, slightly faster than polynomial in $n$) achieve adaptive (near)-minimax rates in the three aforementioned zones. This includes the `sparse' zone with convergence rates featuring indices relative to both the functional class and the loss function. We also derive a sharpness result showing that scalings with exactly polynomial decrease can lead to suboptimal rates;
\item  for the first time, we obtain posterior contraction rates for horseshoe prior distributions in the context of nonparametric estimation; doing so is nontrivial, since this prior both features very heavy (Cauchy) tails and a density going to infinity at zero. We also obtain rates for OT priors without the two moment assumption imposed in \cite{AC}.
\item we implement our method and show that in terms of statistical risk it is at least comparable to, while sometimes improving upon, state-of-the-art software.
\end{enumerate}

\textit{Gaussian white noise model.} For a given true regression function $f_0 \in L_2 := L_2([0,1])$ and $n \geq 1$, the data $Y^{(n)}$ will be generated from the following model
\begin{equation}\label{def : GWN}
    d Y^{(n)} = f_0(t) \, dt + \frac{1}{\sqrt{n}}d W(t), \qquad t\in [0,1],
\end{equation}
where $W$ is the standard Brownian motion on $\R$. Given $\{\varphi_k \, : \, k \geq 1\}$ an orthonormal basis of $L_2$ for the canonical inner product $\langle . \, , \, . \rangle$, we denote by $f_{0,k} \coloneqq \langle f_0 , \varphi_k \rangle$ the coefficients of $f_0$ (which are square summable, $(f_{0,k})\in\ell_2$), so that 
\[ f_0 = \sum_{k\geq 1 } f_{0,k} \varphi_k \qquad \text{ holds in the $L_2$--sense}.\]
Projecting \eqref{def : GWN} onto $\{\varphi_k\}$ and denoting $X_k \coloneqq \int_0^1 \varphi_k(t) \, dY^{(n)}(t)$, one gets the (infinite) normal sequence model
\begin{equation}\label{def : nseq}
    X_k =  f_{0,k} + \frac{1}{\sqrt{n}}\xi_k, \quad k \geq 1,
\end{equation}
where $\{\xi_k\}_{k \geq 1}$ are independent $\mathcal{N}(0,1)$ variables. In the following, we write $X = X^{(n)} = (X_1,X_2, \dots )$ for the corresponding sequence of observations.

\medskip
\textit{Frequentist analysis of posterior distributions.} Consider data $X$ generated from the sequence model \eqref{def : nseq} with true regression function $f_0\in L_2$ identified to its coefficients $(f_{0,k}) \in\ell_2$. In the following, this will be written as $X = X^{(n)} \sim P_{f_0}^{(n)}$, and  one denotes by  $E_{f_0}$  the expectation under this distribution. The reconstruction of $f_0$ from the data $X$ will be conducted through a Bayesian analysis: from a prior distribution $\Pi$ on $\ell_2$ to be chosen below, one constructs a data-dependent probability measure on  $\ell_2$ called the posterior distribution and given, for any measurable $B \subset \ell_2$, by
\begin{equation}\label{def : post}
\Pi[ B \, | \, X ] := \frac{\int_B \exp\{ -\frac{n}{2} \sum_{k\geq1} (X_k -f_k)^2 \} \, d\Pi(f) }{\int \exp\{ -\frac{n}{2} \sum_{k\geq1} (X_k -f_k)^2 \} \, d\Pi(f)}.
\end{equation} 
{For a given distance $d$, we say that the posterior contracts around $f_0$ at the rate $\varepsilon_n \to 0$ in $d$-loss, if 
\begin{equation}\label{def : contrate}
E_{f_0}\Pi(d(f,f_0)\le M\varepsilon_n \, | X) \to 1\quad \text{
as $n \to \infty$},\end{equation} with $M>0$ a sufficiently large constant.}

\textit{Heavy tailed series priors}. For a function $f \in L_2$, with $f = \sum_{k\geq 1 } f_k \varphi_k$, we define a prior $\Pi$ on $f$ by drawing the coefficients $f_k$ independently using heavy-tailed density functions. Here we consider heavy-tailed priors as in \cite{AC} that have the form
\begin{equation}\label{def : priork}
    f_k = \sigma_k \zeta_k,
\end{equation} 
where $\sigma_k$ are deterministic decaying coefficients and ${\zeta_k}$ are independent identically distributed (i.i.d.) random variables with a given heavy-tailed density $h$ on $\R$. In particular, we assume that $h$ satisfies the following conditions: there exist constants $c_1, c_2>0$ and $\kappa \geq 0$ such that
\begin{enumerate}
    \item[\mylabel{H1}{(H1)}]  \text{$h$ is symmetric, positive, and decreasing on $[0,+\infty)$,}
    
    \item[\mylabel{H2}{(H2)}] for all $x \geq 0$,
    \begin{equation*} 
        \log \left( 1/h(x) \right) \leq c_1 (1+\log^{1+\kappa}(1+x)),
    \end{equation*}
    \item[\mylabel{H3}{(H3)}] for all $x \geq 1$,
    \begin{equation*} 
         \overline{H}(x) := \int_x^{+\infty} h(u)du \leq \frac{c_2}{x}.
    \end{equation*} 
\end{enumerate}
 Note that, compared to \cite{AC}, the density $h$ is not required to be bounded and must only satisfy the tail condition (H3).  The latter condition is very mild and includes distributions without integer  moments.  
 For instance, $\kappa=0$ in (H2) allows for densities with polynomial tails such as the Cauchy density.  Note that although the prior might not have a finite mean, the Gaussianity of the likelihood in model \eqref{def : nseq} ensures that all posterior moments, including the posterior mean, are finite.
 
 {\em Oversmoothed heavy-Tailed (OT) priors.} 
 Following \cite{AC}, the choice of  scalings $(\sigma_k)$ in \eqref{def : priork} we argue in favour, is a deterministic decay which is (slightly) faster than any polynomial in $k$:  for  some $\nu >0$ and all $k\ge 1$, we set
\begin{equation}\label{def : OTprior}
   \sigma_k = e^{-(\log k)^{1 + \nu}}.
\end{equation}
 For any specific choice of the density $h$ satisfying assumptions (H1)-(H2)-(H3), and any fixed $\nu>0$, the resulting prior is an instance of what we call the OT-prior. The idea behind this prior is to  shrink small values of the observed noisy coefficients by the deterministic small scaling factors $\sigma_k$, while the heavy-tails of the $\zeta_k$ will enable the prior to capture substantial signals with high probability, see \cite[Section 2]{AC} for a more detailed intuition. The results below show that the posterior contracts at the minimax rate up to log-factors depending on $\kappa$ and $\nu >0$. Although theory suggests taking $\nu$ as small as possible, finite noise precision effects lead us not to choose $\nu$ too small:  since the value of the hyper-parameter $\nu$ has relatively little effect on the non-asymptotic behavior of the posterior, across all simulations presented in Section \ref{sec:sims}, we take the same value $\nu =1/2$.



\medskip

\textit{Student and Horseshoe priors}. We consider two specific classes of  distributions for $\zeta_k$ in \eqref{def : priork}. The first is the class of Student distributions with at least one degree of freedom, that verify (H1)-(H2)-(H3). The second is the Horseshoe prior with parameter $\sigma_k >0$,
\begin{equation}\label{def : HSprior}
  \qquad f_k \sim HS(\sigma_k) \qquad \text{independently across $k \geq1$;}
\end{equation}
for any $\tau >0$, the distribution of $f \sim HS(\tau)$ is a mixture arising from $\lambda \sim C^+(0,1)$ and $f|\lambda \sim \mathcal{N}(0,\tau\lambda^2)$, where $C^+(0,1)$ is the half-Cauchy distribution with location parameter $0$ and scale $1$. The Horseshoe distribution $HS(\tau)$ has a density $h_{\tau}$ satisfying, for all $t \neq 0$ (see \cite{hs1}, Theorem 1.1),
    \begin{equation}\label{eq : HSandwich}
        \frac{1}{(2\pi)^{3/2} \tau} \log(1 + \frac{4 \tau^2}{t^2}) \leq h_{\tau}(t) \leq \frac{1}{(2\pi)^{3/2}\tau} \log(1 +\frac{2\tau^2}{t^2}).
    \end{equation} 
Horseshoe distributions are of the form \eqref{def : priork}, indeed if $f_k \sim HS(\sigma_k)$ one can write $f_k = \sigma_k \zeta_k^1$ where $\zeta_k^1 \sim HS(1)$. A random variable $\zeta^1 \sim HS(1)$ is symmetric and has Cauchy-like tails (this readily follows from \eqref{eq : HSandwich}), its density $h_1$ satisfies conditions (H1), (H2) with $\kappa =0$, and $(H3)$. A variation of the Horseshoe prior that is of common use in high-dimensional models is a truncated version with uniform rescaling (with respect to $k$), that is for some $\tau>0$,
    \begin{equation}\label{def : truncHSprior}
    f_k \overset{i.i.d}{\sim} HS(\tau)  \ \text{  for  }\ k \leq n, \quad \text{and} \quad f_k = 0 \ \text{ for }\  k > n.
    \end{equation}
This prior matches the definition of the heavy-tailed prior \eqref{def : priork} with $\sigma_k=\tau$ for all $k$ up to the truncation point $k = n$, after which it assigns all coefficients to zero. The uniform rescaling $\tau$ is typically taken to be of the order of a negative power of $n$, see the discussion in Section \ref{sec:disc} and Theorem \ref{thm : OTconc} below.\\



\emph{Outline.} In Section \ref{sec : Sobolev}, we establish posterior contraction rates in the $L_2$-norm for Hilbert–Sobolev truths. Specifically, in Section \ref{subsec : sobo upper} we show that OT priors on basis coefficients, as well as truncated horseshoe priors, achieve (nearly) minimax-optimal rates. In Section \ref{subsec : sobo lower}, we derive lower bounds (both in probability and in expectation) for heavy-tailed priors with the slightly more common (so far in the literature of series priors) choice of polynomially decaying scalings of the form $\sigma_k=k^{-\alpha - 1/2}$. In particular, for $\beta$-regular Sobolev truths, while this choice was shown in \cite{AC} to lead to (nearly) minimax rates in the over-smoothing regime $\alpha\ge\beta$, here we show that they give rise to polynomially slower--than--minimax rates in the under-smoothing regime $\alpha < \beta$, hence establishing that they cannot be fully adaptive to smoothness, which gives another rationale for choosing OT-scalings as in \eqref{def : OTprior}.

Section \ref{sec : besov} considers a much broader setting of functional classes allowing for non-homogeneous smoothness and non-matched losses, showing that OT heavy-tailed priors adapt to the Besov regularity of the truth and achieve (nearly) minimax rates for general $L_p$ losses, in all three zones, "regular", "intermediate" and "sparse".

Section \ref{sec:sims} provides numerical experiments along with implementation details, illustrating a range of rate behaviors across different signal classes. Section \ref{sec:disc} provides a brief discussion putting the results of the paper into perspective. Sections  \ref{proof : OTconc} and \ref{proof : lbprob} contain the proofs of Theorem \ref{thm : OTconc} and \ref{thm : lbprob} respectively. All other proofs and additional technical material are contained in the supplement \cite{sm}.

\section{Contraction in $L_2$--loss for Hilbert-Sobolev spaces}\label{sec : Sobolev}
The rate at which the posterior distributions \eqref{def : post} concentrate around the truth $f_0$ as in \eqref{def : contrate}, crucially depends on both the regularity of the truth and the loss $d$. {As a warm-up to the more elaborated results considered in the next section with double-indexed wavelet bases}, we consider in this section the simplest case of Hilbert-Sobolev-type balls and the $L_2$--loss, {with expansions into a simple single-index orthonormal basis (e.g. the Fourier basis)}. Recalling that $f_k = \langle f , \varphi_k \rangle$, for $\beta , F >0$, denote
\begin{equation}\label{def : sobolev ball}
    \mathcal{S}^{\beta}(F) = \Big\{ f \in L_2([0,1]) \, : \quad \sum_{k\geq 1 } k^{2 \beta} f_k^2  \leq F^2 \Big\}.
\end{equation}
For a suitable choice of orthonormal basis $\{ \varphi_k \}$, the set $\mathcal{S}^{\beta}(F)$ corresponds to a ball of $L_2$--functions with $\beta$ derivatives in the  $L_2$--sense. A more general study of posterior contraction under $L_p$--losses and Besov smoothness is conducted in Section \ref{sec : besov}.

\subsection{Upper bounds for OT and truncated Horseshoe priors}\label{subsec : sobo upper}
Our first result provides an upper bound on the contraction rate for OT priors, where the decay sequence is given by \eqref{def : OTprior}. Up to a logarithmic factor (a logarithmic loss is in fact unavoidable for separable estimators, as shown in \cite{cai08}; see Remark \ref{remlogs} below for more on this) such priors achieve the minimax convergence rate in terms of the $L_2$--loss over the above class.

In all the results to follow, when we refer to priors given by \eqref{def : priork}, it will be understood that the prior density $h$ of the variables $\zeta_k$ satisfies  conditions \ref{H1}--\ref{H2}--\ref{H3}.

\begin{theorem}\label{thm : OTconc}
   Suppose, in the sequence model \eqref{def : nseq}, that $f_0 \in \mathcal{S}^{\beta}(F)$ and let $\Pi$ be the OT series prior with independent coefficients $(f_k)$ given by \eqref{def : priork} (in particular, the prior \eqref{def : HSprior} is admissible) and scalings given by \eqref{def : OTprior}. As $n \to \infty$,
    \[ E_{f_0} \Pi \left[ \bigg\{ f \, : \, ||f - f_0||_2^2 >  M_n  \left(\frac{ \log^{(1+\kappa)(1+\nu)}n}{n} \right)^{\frac{ 2 \beta }{2 \beta +1}}  \bigg\} \, \bigg| \, X \right] \to 0 ,\]
    here $M_n$ is an arbitrary sequence such that $M_n \to \infty$ as $n \to \infty$. The same conclusion holds  if the prior $\Pi$ is  the truncated Horseshoe prior \eqref{def : truncHSprior} with uniform scaling $\tau = n^{-a}$ and $a \ge 4$, up to replacing the logarithmic factor by $\log^{ 2\beta / (2\beta +1)} n$ (case $ \kappa =0$ and $\nu =0$).
    \end{theorem}

The proof is given in Section \ref{proof : OTconc}. Theorem \ref{thm : OTconc} establishes that, for the OT priors, the posterior  contraction rate is nearly minimax  over Sobolev balls $\mathcal{S}^{\beta}(F)$. Among special case of OT priors, one can choose for instance a Cauchy density for $h$, as well as  the horseshoe density $h_\tau$; further, this also holds for the truncated version of the horseshoe prior \eqref{def : truncHSprior}. This result strengthens Theorem 1 in \cite{AC}, where a similar upper bound was obtained for OT priors, but with bounded densities and finite second-order moment. The key difference arises from our proof technique: rather than applying Markov's inequality and then working under the posterior expectation, we work directly at the level of the posterior probability. This approach allows us to replace the finite second moment condition by the weaker tail condition \ref{H3}, enabling the use of heavier-tailed priors, which may have unbounded expectation, such as the Cauchy and Horseshoe priors. A simulation study for these heavy-tailed priors is provided in Section \ref{sec : simu OTHS}.  We also note that in \cite{AC} a result for Cauchy priors was provided, but for fractional posteriors only, and not the standard posterior as considered here. 

\begin{remark}[logarithmic factors] \label{remlogs}
According to \cite{cai08}, Theorem 2, separable estimators incur a minimum cost of adaptation in quadratic risk of a factor $(\log n)^{ 2\beta / (2\beta +1)}$  with respect to the minimax rate $n^{-2\beta / (2\beta +1)}$. If one takes a density $h$ with
Student-type tails, then  $\kappa=0$. Since $\nu$ in \eqref{def : OTprior} can be chosen as small as desired, one sees that one can get a logarithmic loss as close as desired to the optimal one which would correspond to $\nu=0$. With the truncated horseshoe prior, we even get this optimal logarithmic factor for separable rules exactly. Removing logarithmic factors altogether requires leaving the class of separable rules, see the Discussion Section below for more on this.
\end{remark}

\begin{remark} When using the truncated Horseshoe in Theorem \ref{thm : OTconc}, we take $\tau=n^{-b}$ with $b=4$. This is for technical reasons, and the proof similarly goes through for any power $b\ge4$. In practice, unless $n$ is very large, one may prefer to take a less conservative choice e.g. $\tau=1/n$ (as we do in the simulations section). We conjecture that the conclusion of Theorem \ref{thm : OTconc} still holds true for this choice. \end{remark}

\subsection{Lower bound for polynomially decaying scalings}\label{subsec : sobo lower}
{Another possible choice of prior to recover $f_0 \in \mathcal{S}^\beta(F)$ is the HT($\alpha$) prior, which is defined as follows. It is a heavy-tailed prior \eqref{def : priork} defined in a similar way as the OT prior, but with {\em polynomially} decaying scaling parameters, for some $\alpha>0$, }
\begin{equation}\label{def : HTprior}
    \sigma_k = k^{-1/2 - \alpha}, \qquad\text{for all } k\geq 1.
\end{equation}
In \cite{AC}, it was shown that this prior leads to adaptation in the over-smoothing case, that is, if $\alpha \ge \beta$,  the posterior contraction rate is the minimax rate over $ \mathcal{S}^\beta(F)$ (up to log factors). If on the contrary, it turns out that $\alpha< \beta$, then since parameter classes such as $\mathcal{S}^{\beta}(F)$ are nested (i.e. $\mathcal{S}^{\beta}(F)\subset\mathcal{S}^{\alpha}(F)$ if $\alpha<\beta$), we have that $f_0\in \mathcal{S}^{\alpha}(F)$, so by applying the over-smoothing result with $\alpha$ in place of $\beta$ (we can since we are now in the case of `matched' regularity of prior and truth), one gets that the posterior contracts at rate (at least) $n^{-2\alpha/(2\alpha+1)}$ up to a log factor in terms of the quadratic risk. However, this does not rule out that the rate is in fact faster. Our next result shows that this is not the case by providing a matching lower bound for the posterior contraction rate in the under-smoothing case $\alpha < \beta$.
\begin{theorem}\label{thm : lbprob}
    Let $f_0 \in \mathcal{S}^{\beta}(F)$ and $\Pi$ be the $\text{HT} \,(\alpha)$ series prior with $ 1/2 < \alpha < \beta$. For $ \delta >0$, define
    \[ \varphi_n :=   n^{- \alpha /(2 \alpha +1)} \log^{-\delta}n.   \]
    Then, for $\delta$ a large enough constant, as $n \to \infty$,
    \[ E_{f_0} \Pi \left[ \{ f \, : \, \| f - f_0 \|_2 \leq \varphi_n \, | \, X\right] \to 0.\]
    Under the weaker condition $0 < \alpha  < \beta,$ the following in-expectation bound holds, as $n \to \infty$,
\[ E_{f_0} \int \| f - f_0 \|_2^2 \, d\Pi(f \, | \, X) \gtrsim  n^{-2 \alpha /(2 \alpha +1)}.\]
     
\end{theorem}
The proof is provided in Section \ref{proof : lbprob} and Section C in the supplement \cite{sm}. Theorem \ref{thm : lbprob} shows that in the under-smoothing case $\alpha < \beta$, the posterior cannot concentrate at a rate faster than $n^{-\alpha/(2\alpha+1)}$, which is slower than the minimax rate over $\mathcal{S}^\beta(L) $. This is a sharpness (i.e. lower-bound) result that demonstrates that the adaptive property of the HT$(\alpha$) prior is fundamentally one-sided, meaning that $\alpha$ must be chosen large enough to ensure adaptation. This fact provides a strong incentive to use the OT prior \eqref{def : OTprior}, which, as shown in Theorem~\ref{thm : OTconc}, always guarantees adaptation by ensuring that one is `always in the over-smoothing case' (hence the name). Also, note that the lower bound result holds for any given function in $\mathcal{S}^{\beta}(F)$, and not just for some `least-favourable' ones.
 
The first part of Theorem \ref{thm : lbprob} provides a lower bound for the in-probability contraction rate, which is expressed in the same form as the in-probability upper bound in Theorem \ref{thm : OTconc}. By an application of a Markov-like inequality, it is not hard to check that this result leads to a lower bound for the contraction rate `in expectation' $E_{f_0} \int \| f - f_0 \|_2^2 \, d\Pi(f \, | \, X)$, with a rate $\varphi_n^2$ (including a logarithmic factor). The second part of the statement provides a more precise lower bound, without logarithmic factor.  
Simulations regarding contraction in the under-smoothing case are provided in Section A in the supplement \cite{sm}.

\begin{remark}\label{rem:lbass} 
The additional prior condition $\alpha > 1/2$ appearing in the in-probability lower bound arises from our proof technique, based on Lemma \ref{lem : aurzada} which assumes that the density $h$ in our definition of the heavy-tailed series priors possesses certain fractional moments. An inspection of the proof of Theorem~\ref{thm : lbprob}, reveals that this condition can be relaxed to $\alpha>(1/\rho-1/2)\vee 0$ if we slightly strengthen assumption \ref{H3} to $\overline{H}(x) \lesssim x^{-\rho}$ with $\rho > 1$. For instance, for a Student distribution with at least two degrees of freedom, the prior regularity condition becomes empty. 
\end{remark}

\section{Sparse Besov rates using heavy-tailed priors}\label{sec : besov}
In this section we consider the problem of estimating, from the Bayesian nonparametrics point of view,  a Besov function in $L_{p'}$--losses for $p'\geq1$ in the white noise model, again understood as projected into sequence space; we wish to achieve this through a method (prior) that leads to a posterior distribution {\em adaptive} to the smoothness level of the unknown regression function. To the best of our knowledge, this question has not been addressed for a Bayesian method (or, more generally, for any likelihood-based approach). We first recall the definition of Besov spaces in terms of wavelet coefficients and then introduce the statistical problem and core result of this paper.

\subsection{Besov classes, notation}

Let $S >0$, and $J$ be the smallest integer satisfying $2^J \geq 2S$. Consider $\{ \phi_{Jk} \, : \, k \in \{0 , \dots , 2^J-1\}\} \cup \{ \psi_{jk} \, : \, j \geq J , k \in \{0 , \dots , 2^j-1\} \} $ an orthonormal, boundary corrected, $S$-regular wavelet basis of $L_2([0,1])$ (see \cite{cohen} for a construction). For any $f \in L_p$ (if $p < \infty$) or $f$ continuous (if $p = \infty$), we have,
\[ f = \sum_{k =0}^{2^J -1} \langle f, \phi_{Jk}  \rangle \phi_{Jk}+ \sum_{j \geq J} \sum_{k =0}^{2^j-1} \langle f, \psi_{jk}  \rangle \psi_{jk} \quad \text{in } L_p, \quad 1\leq p \leq \infty. \]
To simplify the notation and to avoid splitting the study of the scaling and wavelet coefficients in the following statistical results, whenever $j \geq J$, we write $f_{jk} := \langle f , \psi_{j k} \rangle $ and whenever $j < J$, we write
\begin{align*}
    f_{jk} &:= \langle f , \phi_{J 2^{j}+k} \rangle , \quad 0 \leq j < J,\quad 0\leq k < 2^j, \\
    f_{-1,0} &:= \langle f , \phi_{J 0} \rangle.
\end{align*}
The $L_p$--norm of the sequence $(f_{jk})_k$ with $k \in \{0 , \dots , 2^j-1 \}$ is denoted as
\begin{equation*}
    ||f_{j \cdot}||_p := \left(\sum_{k=0}^{2^j-1} |f_{jk}|^p\right)^{1/p}.
\end{equation*}
Under conditions on the wavelet basis (all satisfied with a boundary corrected $S$-Daubechies system), we have a Parseval-like quasi-equality for $L_p$--norms, with $p \geq 1$, (see Proposition 4.2.8 in \cite{GineNickl} and a modification thereof for boundary corrected bases p.357)
\begin{equation}\label{parseval}
    \left\lVert \sum_{k =0}^{2^j-1} f_{jk} \psi_{jk} \right\rVert_p \simeq  2^{j(1/2-1/p)} ||f_{j \cdot}||_p.
\end{equation}
Here `$\simeq$' means equality up to a constant (depending only on $p$ and the wavelet system). To describe the regularity of a function $f$ we  use the decay of its coefficients in \eqref{parseval}, for $p,q \in [1, \infty )$ and $0 < s < S$, and define the norm
\begin{equation}\label{def : besovnorm}
    ||f||_{B_{pq}^s} :=  \left( \sum_{j \geq -1} 2^{qj(s + 1/2 - 1/p)} ||f_{j \cdot}||_p^q \right)^{1/q},
\end{equation}
with the usual adaptation whenever $p = \infty$ or $q = \infty$. A Besov-type ball is defined for any $p,q \in [1, \infty]$, $F>0$ and $0<s<S$, as
\begin{equation}\label{def : besovball}
    B_{pq}^s(F) = B_{pq}^s([0,1],F) := \{ f \in L_p([0,1]) \, : \, ||f||_{B_{pq}^s} \leq F \}.
\end{equation}
The corresponding Besov space $B_{pq}^s$ is defined as
\begin{equation}\label{def : besovspace}
    B_{pq}^s := \bigcup_{F > 0} B_{pq}^s(F).
\end{equation}
If the wavelet system is chosen smooth enough, the space $B_{pq}^s$ corresponds to the usual Besov space defined in terms of moduli of continuity (see 4.3 in \cite{GineNickl} for a proof). From the definition of Besov spaces in terms of wavelet coefficients as in \eqref{def : besovnorm}, one can deduce the following embedding properties
    \begin{align}\label{eq : embeddingdual}
        B_{pq}^s &\subset B_{p'q}^{s'} \quad \text{whenever } p' > p \text{ and } s'-1/p' = s - 1/p ,\\
        \label{eq : embeddingtrivial}
        B_{pq}^{s} &\subset B_{pq'}^{s'} \quad \text{whenever } s' < s \text{ or when } s'=s \text{ and } q' \geq q.
    \end{align}
    In particular, if $s - 1/p >0$, $B_{pq}^s \subset B_{\infty \infty}^{s'}$ is included in the space of continuous functions. 

\subsection{Regression in Besov spaces, sparse rates, contraction for OT decay}

Following the same approach as in the Hilbert-Sobolev case, if the signal belongs to a Besov space $B_{pq}^s$, from the definition of the space in terms of wavelet coefficients it is natural to estimate it from the projections of the white noise model into the (now, double-indexed wavelet) basis, that is,
\begin{equation}\label{def : nseqwav}
    X_{jk} = f_{0,jk} + \frac{1}{\sqrt{n}} \xi_{jk}, \quad j\geq -1, \quad 0 \leq k <2^{j},
\end{equation}
where $\{ \xi_{jk}\}$ are independent $\mathcal{N}(0,1)$ variables. The next condition ensures that $B_{pq}^s  \subset L_{p'}$ so that the problem of estimating the true unknown function $f_0 \in B_{pq}^s$ in $L_{p'}$--norm is well-defined (also, recall that we only consider positive smoothness indices $s>0$) 
\begin{equation}\label{indices}
s - 1/p + 1/p'>0.
\end{equation}
Minimax rates in this sequence setting and Besov smoothness were derived by Donoho and Johnstone \cite{djkp97} (announced in \cite{djkp95})
, building up from the seminal work of Nemirovskii \cite{nem85}; the case of the density estimation model was treated in \cite{djkp96}. We refer to the monograph \cite{hkpt98}, Section 10.4, for a detailed discussion. As it turns out, the minimax rate for this estimation problem crucially depends on the parameter $p'$ of the  loss function. To describe this rate, we now introduce some notation -- we note that for simplicity we will not keep track on the precise power in the logarithmic factors appearing in the rates; this slightly simplifies the discussion compared to \cite{djkp97} (who also treat the even more general case of Besov losses)--.

 We particularize two regions where the statistical behavior differs
\begin{align*}
    \mathcal{R} :=  \{(p,p') \, : \,  p' < (2s +1)p  \} \qquad \text{and}
 \qquad \mathcal{S} := \{ (p,p') \, : \,  p' \geq (2s+1)p \}.
 \end{align*}

The following index $\eta$ specifies in which region the indices 
$(p,p')$ are
\begin{equation}\label{def : indexeta}
    \eta := sp - \frac{p'-p}{2},
\end{equation}
as indeed $\eta>0$ if and only if $(p,p')\in \mathcal{R}$, and $\eta \leq 0$  if and only if  $(p,p')\in\mathcal{S}$. 
 Consider
\begin{equation}\label{def : s'}
    s' := s -1/p +1/p',
\end{equation}
which is positive by \eqref{indices}, and define the rate $\varepsilon_n := n^{-r}$, where, for $\eta$ as in \eqref{def : indexeta},
\begin{equation}\label{def : rateexp}
    r := \begin{cases}
    s / (1 + 2s), &\quad \text{if $\eta > 0$} \\
    s' / (1 + 2(s -1/p)), & \quad  \text{if $\eta \leq 0$}
\end{cases}.
\end{equation}

In the following lines, we provide some insight into the appearance of the elbow in the rate \eqref{def : rateexp} and the distinction between the two regions. Note that this rate is known to be minimax (up to logarithmic factors) over $B_{pq}^s$ in the continuous case $s > 1/p$. We remark that functions in $B_{pq}^s$ can be "non-homogeneously" smooth; that is, as $p'$ increases, it becomes possible to induce increasingly large perturbations of their $L_{p'}$--norm using only a small number of wavelet coefficients. Consequently, the statistical problem of reconstructing $f_0 \in B_{pq}^s$ from noisy observations of its wavelet coefficients becomes more difficult as $p'$ increases. The region $\mathcal{R}$ is referred to as the {\em regular} region: for $(p,p')$ in $\mathcal{R}$, the rate \eqref{def : rateexp} corresponds to the standard minimax rate encountered in nonparametric statistics. It is worth noting also that $\mathcal{R}$ can further be divided into two zones. The first one is the `homogeneous' zone ${p' \leq p}$, where the aforementioned perturbative effect does not occur; therein, {\em linear} estimators, understood as linear functionals of the empirical measure $n^{-1} \sum \delta_{{X_i}}$, are minimax-optimal. The second zone, called `intermediate' or also `non-homogeneous regular' has $p < p' < (2s+1)p$, the perturbative effect does take place; therein, linear estimators are known to be suboptimal; in fact, the linear-minimax rate here is polynomially slower than the global minimax rate, and is of order $n^{-s'/\{2s'+1\}}$, for $s'$ as defined in \eqref{def : s'}. When $p' < (2s + 1)p$, the perturbative effect is not strong enough to cause a discrepancy in the global minimax rate. In contrast, in the so-called {\em sparse} region $\mathcal{S}$, defined by $p' \geq (2s + 1)p$, the minimax rate becomes polynomially slower than the usual nonparametric rate from the regular region (although still faster than the corresponding linear-minimax rate for $p' < \infty$). In $\mathcal{S}$, the most difficult functions in $B_{pq}^s$ to estimate in the $L_{p'}$-norm exhibit highly localized irregularities, corresponding to only a few wavelet coefficients of large magnitude (hence the term "sparse"), which makes the statistical problem significantly harder.

\begin{theorem} \label{thm : besov sparse}
    Let $0 < s < S$, $p,q \in [1,\infty]$, $1 \leq p' < \infty$, $F> 0$ and suppose \eqref{indices}. 
     
    Consider observations from model \eqref{def : nseqwav} with an unknown function $f_0 \in B_{pq}^{s}(F)$ for some $F>0$. 
    Let $\Pi$ be the OT wavelet series prior sampling coefficients $f_{jk}$ independently as
    \begin{equation} \label{priorcodouble}
    f_{jk} = 2^{-j^{2}} \zeta_{jk} ,
    \end{equation}
     where $( \zeta_{jk} )$ are i.i.d. copies of a heavy-tailed random variable $\zeta$ satisfying conditions (H1)--(H2)--(H3). Then, for $r$  given by \eqref{def : rateexp},
        \[E_{f_0} \Pi\left[ \{ f \, : \, \|f-f_0 \|_{p'} \geq \mathcal{L}_n n^{-r} \} \, | \, X \right] \to 0,\]
as $n \to \infty,$    where $\mathcal{L}_n = \log^{\delta} n$ for some $\delta >0$.
\end{theorem}

The proof of this result is provided in Section D in the supplement \cite{sm}. Theorem \ref{thm : besov sparse} shows that OT series priors achieve complete minimax adaptation, up to logarithmic factors, over Besov-type balls {\em simultaneously} over the regular ($\mathcal{R}$) and sparse ($\mathcal{S}$) regions. To the best of our knowledge, this is the first Bayesian procedure proven to achieve such rates. In contrast, random Gaussian series, similar to linear estimates, are known to be suboptimal in this setting (see \cite{adh21}, Theorem 4.1). Our results also significantly extend  previous work  \cite{AC} and \cite{ce24} on heavy-tailed priors, which beyond classical Sobolev/H\"older spaces considered only the case of Besov spaces $B_{pp}^s$ and anisotropic Besov spaces $B_{pp}^{\mathbf{s}}([0,1]^d)$ (this time in $d$ dimension) with $p < p' = 2$, in particular being restricted to the inhomogeneous part of the regular region $\mathcal{R}$, where linear minimax rates are slower than global minimax rates, but the latter is the usual rate.

\begin{remark}
    The scaling coefficient $2^{-j^2}$ in \eqref{priorcodouble} is the analogue, in double-index notation, of the scaling $e^{-(\log{k})^{1+\nu}}$, with $\nu=1$ therein, and up to a constant factor in the exponent. As anticipated below \eqref{def : OTprior}, and similarly to the single-index OT-prior  considered in Section \ref{sec : Sobolev}, the results of Theorem \ref{thm : besov sparse} still hold, with a similar proof, with scaling coefficients $2^{-j^2}$ in the two-index version \eqref{priorcodouble} replaced by $2^{-j^{1 + \nu}}$ for some $\nu > 0$, although for simplicity to keep the notation minimal, we provide the proof only for the former. In the same spirit, in Theorem \ref{thm : besov sparse} we do not track the logarithmic factors as precisely as in Theorem \ref{thm : OTconc}. While we believe that such factors could be made explicit with minor modifications of the proof (introducing log terms in $J_0$ and $J_1$), we do not pursue this refinement here in order to keep the argument as simple as possible.
\end{remark}

\begin{remark}
For simplicity in the present paper we do not consider the case of the supremum norm ($p' = \infty$). Let us briefly sketch how one can extend our results to this case, following the approach for the supremum norm introduced in \cite{AC} for H\"older-type spaces. 
  For any $(p,q)$, the embedding $B_{p,q}^s \subset B_{\infty,\infty}^{s'}$, given by \eqref{eq : embeddingdual} and \eqref{eq : embeddingtrivial} shows that estimating a $B_{pq}^s$ function in $L_{\infty}$-norm is not harder than estimating a $B_{\infty \infty}^{s'}$ function. Noting that $B_{\infty \infty}^{s'}$ corresponds to the usual Hölder-Zygmund space whenever $s' \notin \mathbb{N}$, one can use techniques of \cite{AC} Theorem 4 (contraction in sup-norm for Hölder-smooth functions) to obtain the desired rate. This rate corresponds, {\em up to a logarithmic factor}, to $n^{-s'/(2s'+1)}$,  regardless of $p$ and $q$ (since $p' = \infty$, we have $s' = s-1/p$, which is positive by \eqref{indices}, and $\eta < 0$ for all $p \geq 1$, thus the target rate - up to log terms - defined in \eqref{def : rateexp} is $n^{-s'/(2s'+1)}$). Note that following the argument in  \cite{AC} will require the additional moment assumption $E|\zeta| < \infty$; we  believe that this condition could possibly be omitted albeit under a slightly more technical argument. 
\end{remark}

\section{A simulation study} \label{sec:sims}
In this section, we present numerical simulations that corroborate and illustrate our theory. {Additional simulations illustrating the lower bound obtained in Theorem \ref{thm : lbprob} are  presented in Appendix A \cite{sm}. }

\subsection{White noise model regression under OT and Horseshoe priors}\label{sec : simu OTHS}

We consider the white noise regression model \eqref{def : GWN}, expanded in the orthonormal basis $\varphi_k(t)=\sqrt{2}\cos(\pi(k-1/2)t)$ 
leading to the normal sequence model \eqref{def : nseq}. As underlying truth, we use a function with coefficients with respect to $(\varphi_k)$ given by $f_{0,k}=k^{-3/2}\sin(k)$. In particular, this true function can be thought of as having Sobolev regularity (almost) $\beta=1$. 

We consider four priors on the coefficients of the unknown. The first three are of the form $f_k=\sigma_k\zeta_k$ for i.i.d. $\zeta_k$, for the same choice of scalings  $\sigma_k=e^{-(\log k)^{3/2}}$, but different   distributions of $\zeta_1$:
\begin{itemize}
\item Student distribution with 3 degrees of freedom, leading to a Student OT prior;
\item Cauchy distribution, leading to the Cauchy OT prior;
\item Horseshoe distribution, leading to the Horseshoe OT prior.
\end{itemize}
The fourth prior is the truncated Horseshoe prior as in \eqref{def : truncHSprior}, with $\sigma_k=\tau=1/n$, where $n$ is the noise precision parameter in \eqref{def : nseq}.

Note that, throughout our simulations for the OT-prior we use the exponent $\nu=1/2$ in \eqref{def : OTprior}. While our results in Section \ref{sec : Sobolev} show that, at least under Sobolev regularity of the truth, choosing a smaller $\nu>0$ improves the logarithmic terms in the rate of contraction (the polynomially decaying terms remain unchanged), in reality and for finite noise precision levels $n$, it is expected that for too small values of $\nu$ the performance of the posterior will deteriorate, as the prior becomes increasingly similar to an HT($1/2$) prior, which our results show leads to suboptimal contraction rates for $\beta>1/2$. Therefore, we only tried the choices $\nu=1/2$ and $\nu=1$ which led to similar results with a slight edge for the former, hence our choice.

Due to independence, the posterior decomposes into an infinite product of univariate posteriors. For all considered priors, we use Stan, with random initialization uniformly on the interval $(-2,2)$, to sample each of the univariate posteriors \cite{stan}. One could also use simple random walk type algorithms, however, Stan is particularly convenient due to its simple implementation and adaptive tuning. In all four cases, we truncate at $K=200$, which, for the considered regularities of the truth and the priors, suffices for the truncation error to be of lower order compared to the estimation error, for the considered noise levels, which range from $n=10^3$ to $n=10^5$ (in the specific case  of the {truncated}  Horseshoe prior case with $\tau_k=1/n$, this truncation point in fact appears somewhat earlier than the truncation point $K=n$ {allowed} in the last part of Theorem \ref{thm : OTconc}).

In Figure \ref{fig-hs}, we present {the posterior means} as well as 95\% credible regions for various noise levels, computed by taking the 95\% out of the 4000 draws (after burn-in/warm up) which are closest to the mean in $L_2$--sense. The three OT priors appear to perform very well at all noise levels, with the two heavier tailed priors (Cauchy and Horseshoe) leading to slightly broader credible regions. The Horseshoe prior with scaling $\tau=1/n$,  appears to be slightly overconfident in all but the lowest noise levels. This was in fact expected based on the intuition on scaled heavy-tailed priors outlined in \cite[Section 2]{AC} in the univariate normal mean model: for small scalings, posterior means {tend to strongly shrink towards $0$}, while for large scalings they tend to preserve the data. Since for this prior, all coefficients (even in low frequencies) correspond to a small scaling $\tau=1/n$, all observations in frequencies with small signal (how small only depends on $n$, and is independent of the frequency), are {set very close to $0$}. As a  result there is very little variance in the posterior. On the contrary, the OT-type priors' scalings for the first few frequencies remain large, before the faster than polynomial decay {in the scaling parameters} eventually kicks in, leading to important frequencies  getting significant values under the posterior (see also  \cite[Section 4]{AC} for more discussion on this).  
This in turn leads to posteriors exhibiting more variability in the OT priors case.

\begin{figure}
    \centering
     \includegraphics[width=0.97\textwidth]{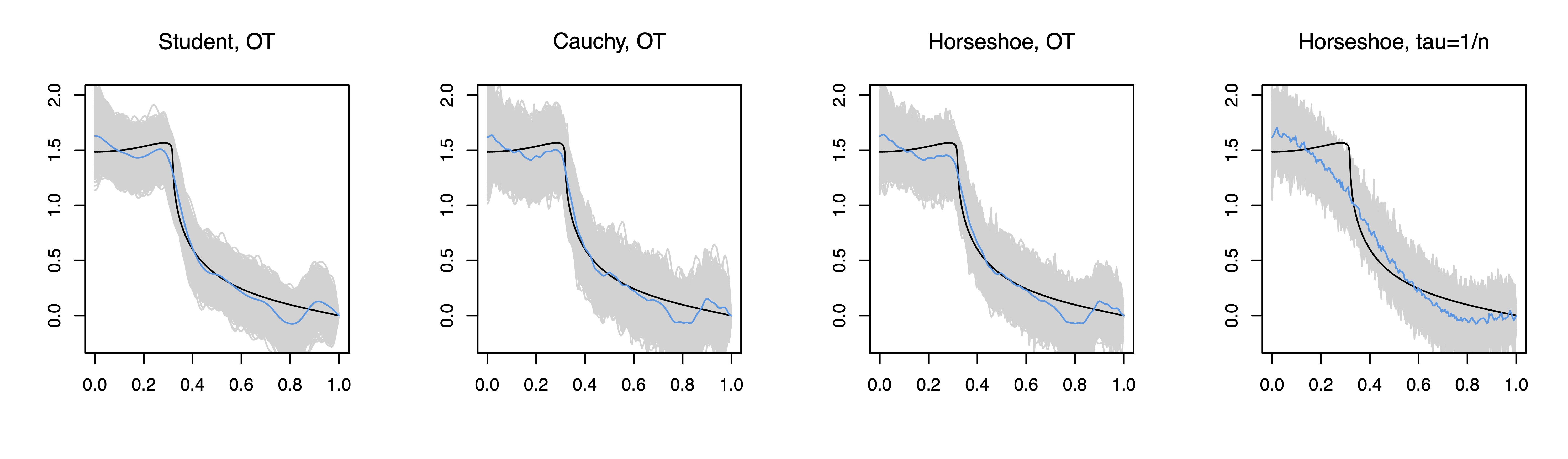}\\
    \includegraphics[width=0.97\textwidth]{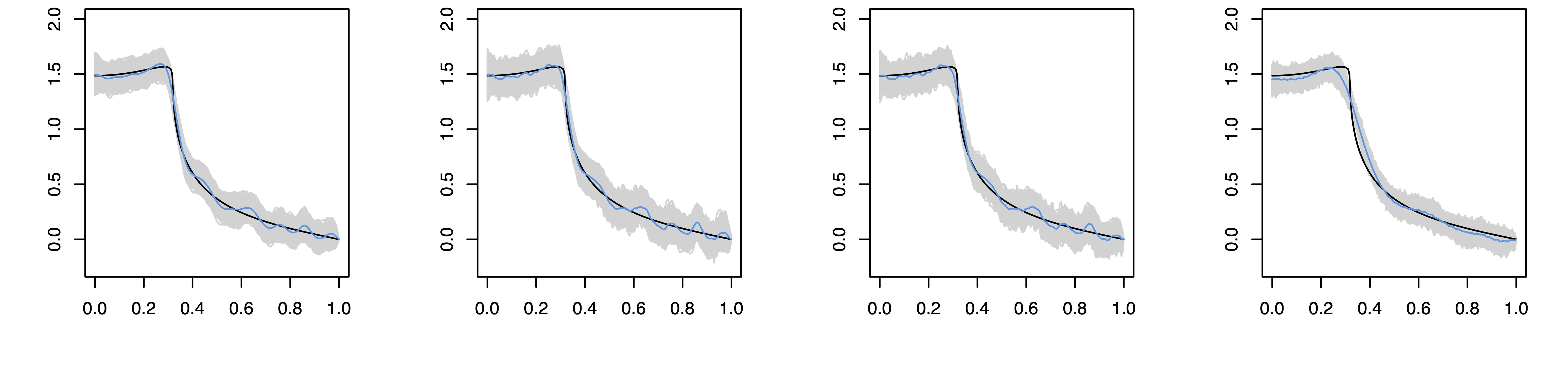}\\
    \includegraphics[width=0.97\textwidth]{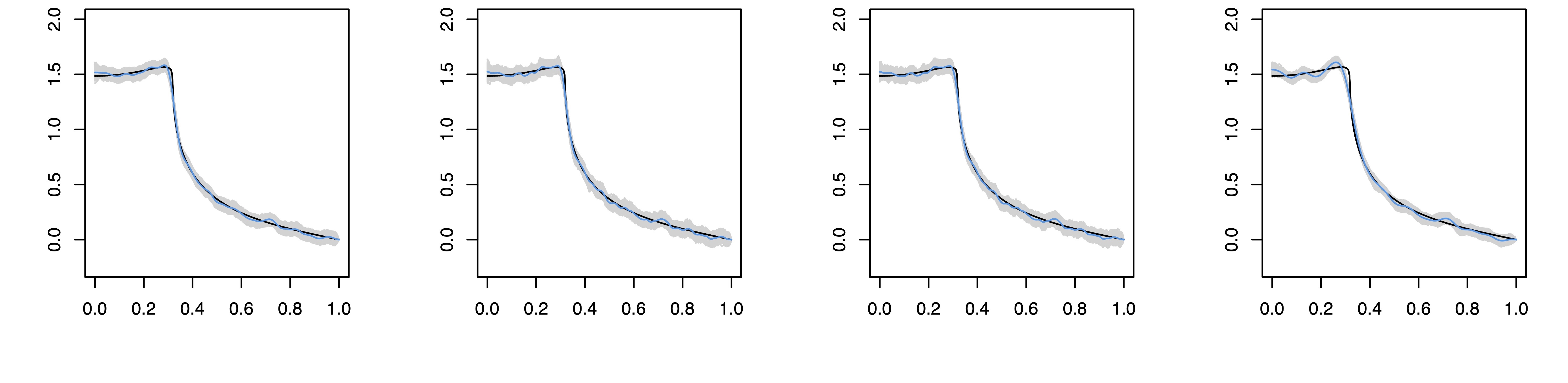}\\
    \caption{White noise model: true function (black), posterior means (blue), 95\% credible regions (grey), for $n=10^3, 10^4, 10^5$ top to bottom and for the four considered priors left to right.}
    \label{fig-hs}
\end{figure}

\subsection{Performance for spatially inhomogeneous truths when varying the loss} \label{sec : simu losses}

In this section we consider four spatially inhomogeneous true functions introduced in \cite{DJ94} for testing wavelet thresholding algorithms. While these functions are not constructed to be `least favourable' in a certain Besov class (we do this in the next subsection), they are prototypes of spatially inhomogeneous signals encountered in practice. Hence, even though not designed to illustrate the change in the performance of the studied priors in terms of error norms when crossing different zones (regular, intermediate and sparse), these functions can illustrate the suitability of the studied priors for typical spatially inhomogeneous functions. We consider the full range of $L_{p'}$--losses, thereby considerably extending the simulation setting for OT priors in \cite[Section E.2]{AC}, where only $L_2$--loss was considered.

The four functions can be seen in Figure \ref{fig-dj94}. We expand the functions in the Daubechies-8 maximally symmetric wavelet basis \cite{daubechiesbook} (Symmlet-8) and add standard normal noise on each wavelet coefficient. We use 2048 coefficients and coarse level 5, while each function has been appropriately rescaled to get a signal-to-noise ratio (as captured by the ratio of the $L_2$--norms of the function to the noise) approximately equal to 7, as in \cite{DJ94}. Hence, we have a normal sequence model as in \eqref{def : nseqwav} with $n=1$. Analysis and synthesis of the wavelet expansions is performed in Wavelab850 \cite{wavelab}.

\begin{figure}[htbp]
    \centering
    \includegraphics[width=0.4\textwidth]{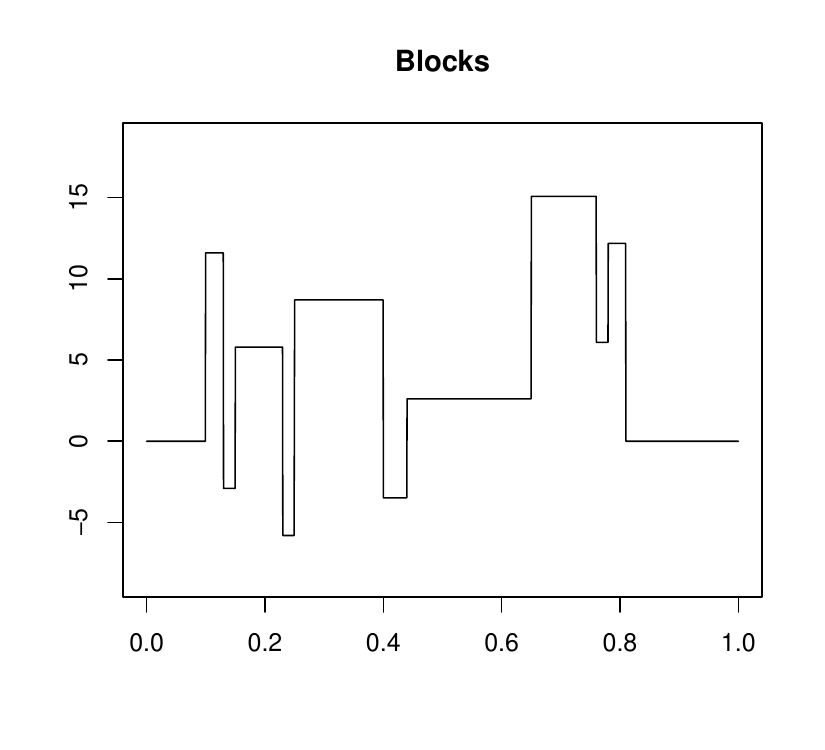}
    \includegraphics[width=0.4\textwidth]{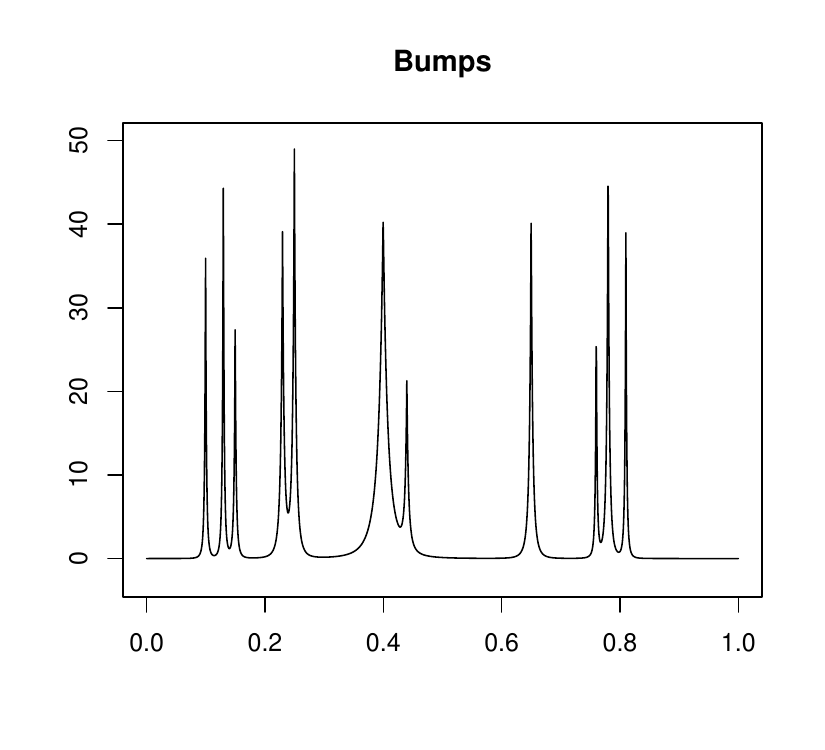}\vspace{-.7cm}
     \includegraphics[width=0.4\textwidth]{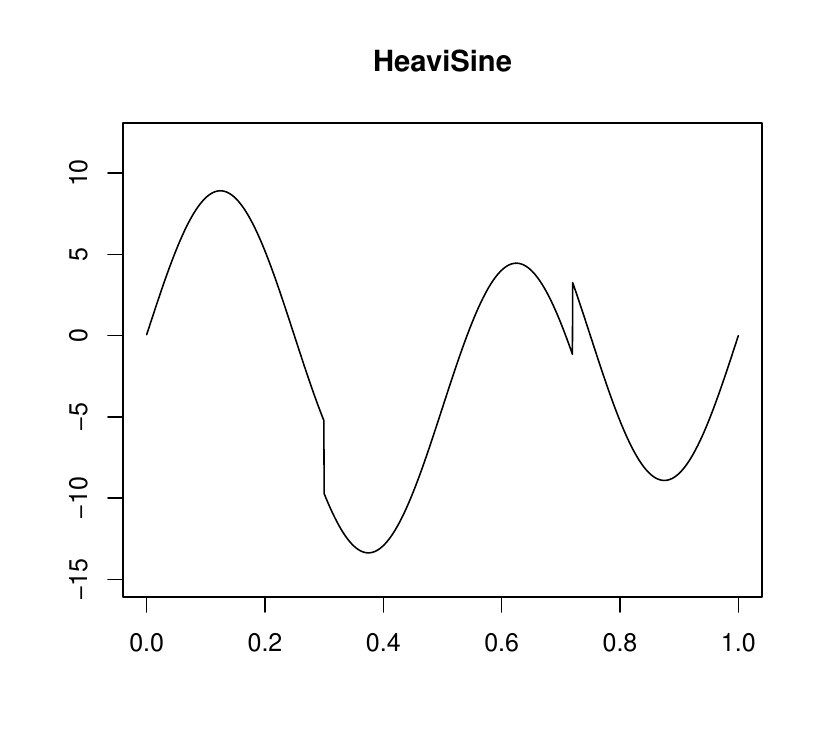}
 \includegraphics[width=0.4\textwidth]{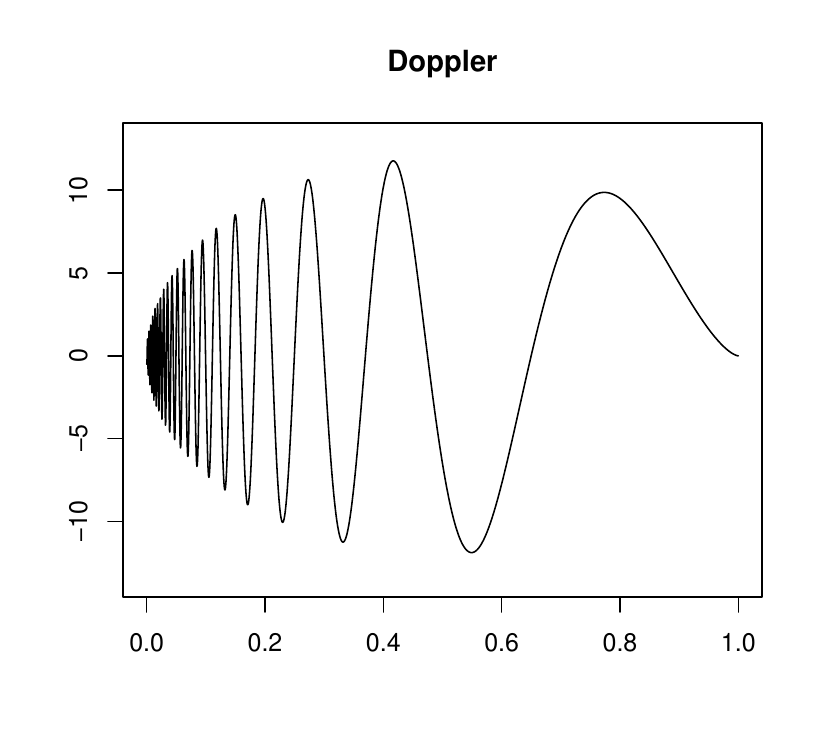}\vspace{-.5cm}
    \caption{Spatially inhomogeneous true functions.}
    \label{fig-dj94}
 \end{figure}
 
Identically to \cite[Section E.2]{AC}, we consider priors on the wavelet coefficients of the form $f_{jk}=\sigma_j\zeta_{jk}$ for i.i.d. $\zeta_{jk}$, with the following choices of the scalings $\sigma_{j}$ and the distribution of $\zeta_{00}$:
\begin{itemize}
\item Gaussian hierarchical prior: $\sigma_j=\tau 2^{-j(1/2+\alpha)}$ with $\tau\sim\text{Inv-Gamma}(1,1)$, $\alpha\sim {\rm Exp}(1)$, $\zeta_{00}$ standard normal;
\item Cauchy OT prior: $\sigma_j=2^{-j^{1+\nu}}$, with $\nu=1/2$, $\zeta_{00}$ distributed according to the standard Cauchy distribution.
\end{itemize}
In addition, we consider frequentist estimation using the hybrid version of SureShrink, which is a soft wavelet thresholding algorithm developed by Donoho and Johnstone to be optimally smoothness-adaptive over Besov spaces, even for extremely sparse signals, see \cite{DJ95, DJ94}. We comment about the Gaussian hierarchical prior that any sequence of Gaussian process priors is known to be limited by the \emph{linear} minimax rate (at least in square loss) which is suboptimal except in the regular homogeneous regime, \cite[Theorem 4.1]{aw21}. For $\alpha$-regular Gaussian series priors and $p\leq2$, the linear minimax rate is, for example, attained for regularity $\alpha=\beta-1/p+1/2$, or by appropriately rescaling provided the prior is not too undersmoothing, \cite[Section 5]{adh21}. As our initial experiments using the Doppler truth showed that a (more conventional) Gaussian hierarchical prior with a hyper-prior only on the scaling $\tau$ performs poorly unless the regularity $\alpha$ is in a restricted window, to avoid manually tuning $\alpha$ for each truth, we chose to use a hyper-prior on both parameters $\alpha$ and $\tau$.

We refer to \cite[Section E.2]{AC} for details on the Markov chain Monte Carlo algorithms employed to sample the posteriors, as well as for the resulting visualizations. The implementation of hybrid SureShrink was done using the WaveShrink function of Wavelab850 \cite{wavelab}. 

To estimate the errors of the three considered methods, we averaged errors over 100 realizations of the data. In particular, for the two priors, we consider two types of errors. The first one is the $L_{p'}$--error of the posterior means, hence after averaging we estimate the error 
\[E_{f_0}||\hat{f}-f_0||_{p'},\]
where $\hat{f}$ is the posterior mean. The same type error is computed for $\hat{f}$ being the thresholding estimator. The second type of error, computed only in the two Bayesian settings, estimates
\[E_{f_0}E_{\Pi[\cdot|X]}||f-f_0||_{p'},\]
where the inner expectation is estimated by taking the average of the $L_{p'}$--errors of the Markov chain samples after burn-in, and the outer by averaging over the 100 data realizations. The latter error captures the contraction of the whole posterior around the truth.

In Figure \ref{fig-DJ94errors} we show estimation errors in $L_{p'}$--loss on the left and contraction-type errors on the right, for $p'=1,2,3,4,6$. Errors in supremum-loss are displayed in Table \ref{tab:spinh}. The Cauchy OT prior, convincingly outperforms the Gaussian hierarchical prior in all losses for the more `jumpy' truths (Blocks and Bumps). For the smoother truths (Doppler and HeaviSine), the Cauchy prior has significantly better errors for $L_{p'}$-losses with smaller $p'$ (especially contraction-type errors), while as $p'$ increases the errors become more even, with the Gaussian prior slightly outperforming Cauchy OT in supremum-loss {(this may only be true for some specific signals, as it is unclear whether the considered hierarchical Gaussian procedure is minimax adaptive in supremum norm;  by analogy with the negative result
from [17] mentioned in the introduction, it may be the case that if only one of the two parameters is drawn at random, then the rate may in fact be suboptimal)}. This is despite the fact that, as can be seen in \cite[Figures E.6 and E.7]{AC}, the Gaussian prior fails to denoise the signals and the Cauchy-based posteriors are visually significantly better. The performance of the Cauchy-OT prior matches and often surpasses that of the hybrid SureShrink algorithm. We stress that for the Cauchy-OT prior we did not tune any parameter; in particular, the choice of $\nu=1/2$ was fixed for all experiments in this work. This is contrast to the SureShrink algorithm which requires choosing the threshold level in a data-driven way.

\begin{figure}
    \centering
     \includegraphics[width=0.48\textwidth]{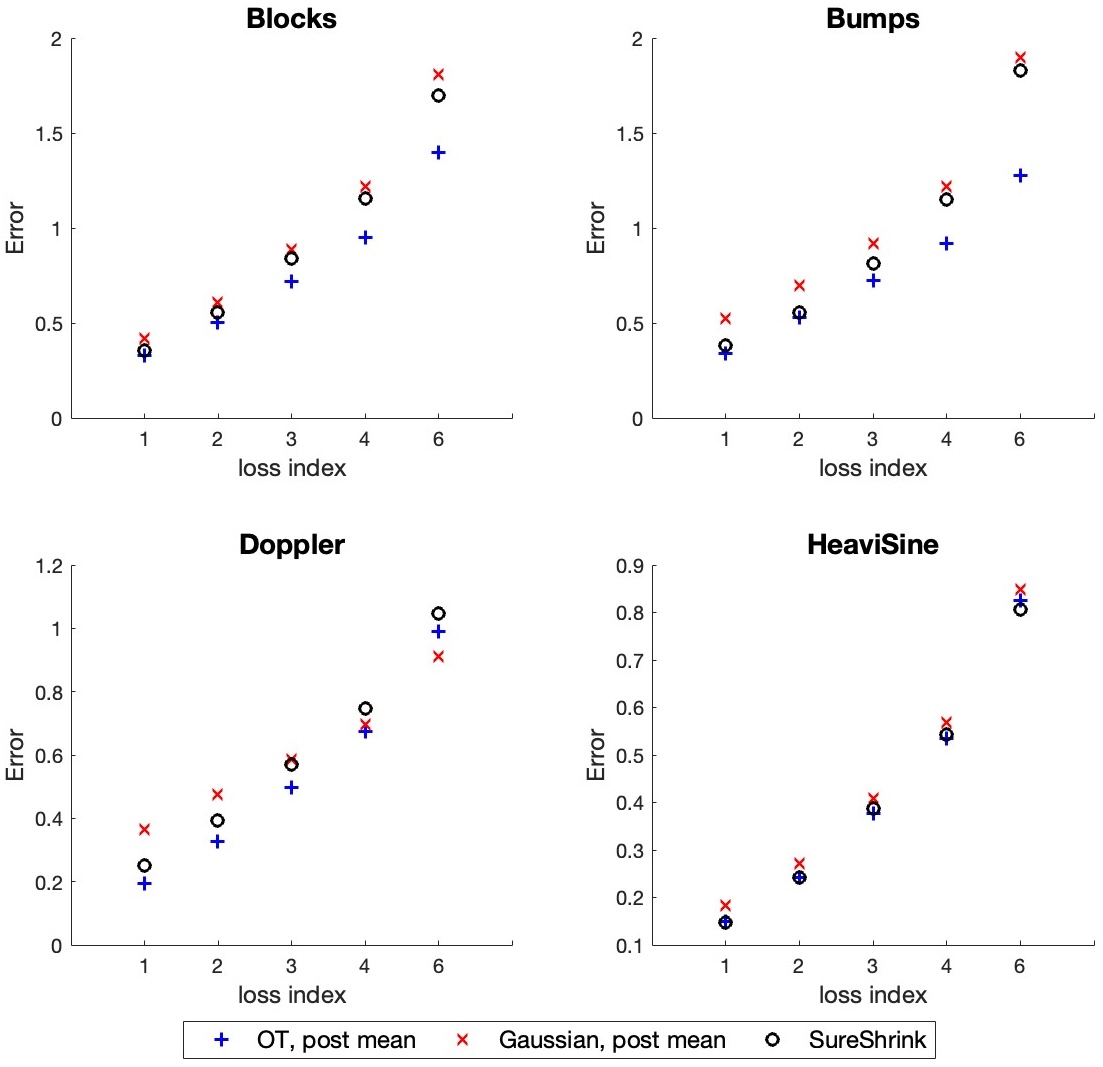}\qquad
     \includegraphics[width=0.48\textwidth]{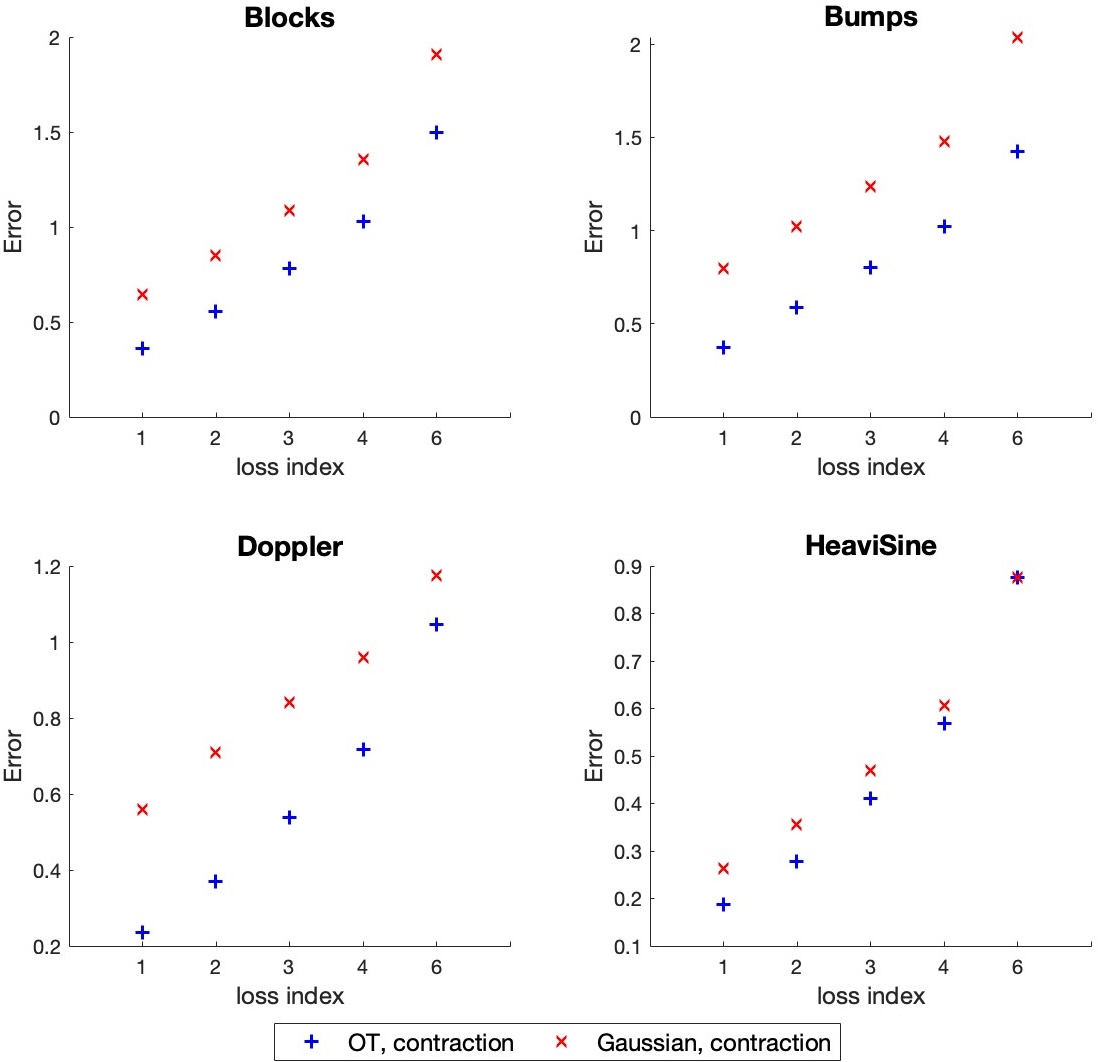}

    \caption{Average errors in $L_{p'}$ for $p'=1,2,3,4,6$, for four model spatially inhomogeneous truths. Signal-to-noise ratio approximately 7 for all truths, errors averaged over 100 realizations of the noise. Errors for posterior means on left, contraction-type errors on right, for Cauchy OT prior (blue plus sign markers) and Gaussian hierarchical prior (red cross markers). In the left plot, the black circles are Hybrid SureShrink estimation errors}
    \label{fig-DJ94errors}
\end{figure}

\begin{table}[ht]
    \centering
  \caption{White noise model with spatially inhomogeneous truths: $L_{\infty}$ average errors of posterior means (contraction-type errors in parentheses), under Cauchy OT and Gaussian hierarchical priors. Hybrid SureShrink estimation error also displayed.}
  \label{tab:spinh}
    \begin{tabular}{llll}
        \toprule
        & \textbf{\quad\hspace{-0.05cm} {Cauchy OT} \quad} & \textbf{Gaussian hierarchical} & \textbf{SureShrink} \\
    \toprule
   Blocks & \quad\; 4.20 (4.46) & \quad\quad\; 5.08 (5.41) & \quad\quad 4.52 \\
   Bumps  & \quad\; 3.65 (4.08) & \quad\quad\; 5.92 (6.21) &  \quad\quad 5.64 \\
   Doppler  & \quad\; 2.85 (3.02) & \quad\quad\; 2.41 (2.91) & \quad\quad 2.86 \\
   HeaviSine  & \quad\; 2.51 (2.69) & \quad\quad\; 2.36 (2.46) & \quad\quad 2.18 \\

        \bottomrule
    \end{tabular}
\end{table}

\subsection{Comparison of the performance of OT priors to wavelet thresholding algorithms, in sparse Besov spaces}

In this subsection, we construct a sequence of true functions belonging in $B^{3/2}_{1\infty}$, which are close to being `least favourable', in the sense that they are the most difficult to estimate among that Besov class. Our construction follows the strategy for constructing the functions used to establish lower bounds on the estimation rate over (sparse) Besov spaces in density estimation, as outlined e.g. in \cite{hkpt98}.

We construct four functions $f_0^{(i)}:[0,1]\to\R, i=1,\dots,4$, defined via their Symmlet-8 wavelet coefficients. Each function, has non-zero wavelet coefficients only at one level: the function $f_0^{(i)}$ has non-zero wavelet coefficients only at the level $j=2i$. Recalling the definition of the $B^{3/2}_{1\infty}$--norm from Section \ref{sec : besov}, this means that each function has $B^{3/2}_{1\infty}$--norm equal to 
\[ 2^{2i}\sum_{k=0}^{2^{2i}-1}|f_{2i,k}|, \quad i=1,\dots,4.\]
For $k=0,\dots, 2^{2i}-1$, we choose
$f_{2i,k}=20\cdot 2^{-2i}w_{2i,k}$, where $|w_{2i,k}|$ sum to 1.
In this way, we ensure that each 
$||f_0^{(i)}||_{B^{3/2}_{1\infty}}=20$, for $i=1,\dots,4$.
The $w_{2i,k}$, are drawn via a `stick-breaking' construction based on uniform random variables, see \cite[Section 3.3.2]{gvbook}; to avoid all significant signal being on the left of the unit interval, we randomly permute the sticks. The resulting functions can be seen in Figure~\ref{fig-badtruths}. Finally, we add standard normal noise scaled by $1/n$ for $n=10^{i+1}$ to the coefficients of each truth, to define the noisy observations. For each truth, we generate 100 realizations of the observation.

We consider a Cauchy OT prior and the hybrid SureShrink estimator. The implementation of the inference is performed in the same way as in the previous subsection. In Figure \ref{fig-badtruthserrors}, we show the logarithms of $L_{p'}$ errors of the posterior means and the thresholding estimators as a function of $\log{n}$. Here, for each $p'=1,2,3,4,6,\infty$, we have four data points, one for each $f_0^{(i)}$, with corresponding noise precision level $n=10^{i+1}$. The errors for each $i$ are averaged over the 100 realizations of the data, before we apply the logarithm. 

According to the minimax rates discussed in Section \ref{sec : besov}, over $B^{3/2}_{1\infty}$, for $p'\le 4$ the minimax rate is the usual rate (here $n^{-3/8}$), while for $p'>4$ we are in the sparse zone and the minimax rate is slower (here $n^{-1/4-1/(2p')}$). Indeed, in Figure \ref{fig-badtruthserrors} we observe that for stronger norms, the errors for both the Cauchy OT prior and SureShrink, appear to decay at a slower rate with $n$. Furthermore, again the performance of the Cauchy OT prior appears to be on par with (hybrid) SureShrink.\\

\begin{figure}
    \centering 
     \includegraphics[width=0.7\textwidth]{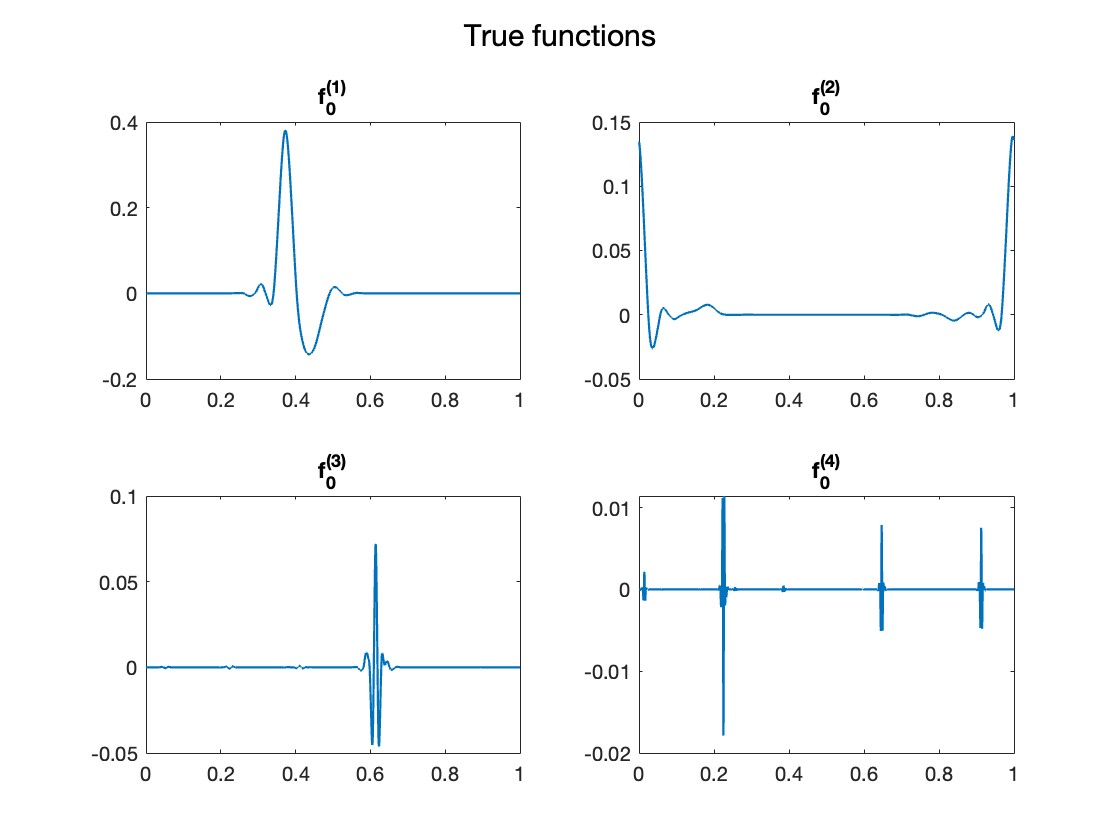}
    \caption{`Least favourable' truths $f_0^{(i)}, \,i=1,\dots,4$ of unit $B^{3/2}_{1\infty}$-norm, constructed to have non-zero wavelet coefficients only at level $j=2i$, and with nonzero coefficients constructed via randomly permuted strick-breaking, scaled by $2^{-2i}$.}
    \label{fig-badtruths}
\end{figure}

\begin{figure}
    \centering
     \includegraphics[width=0.46\textwidth]{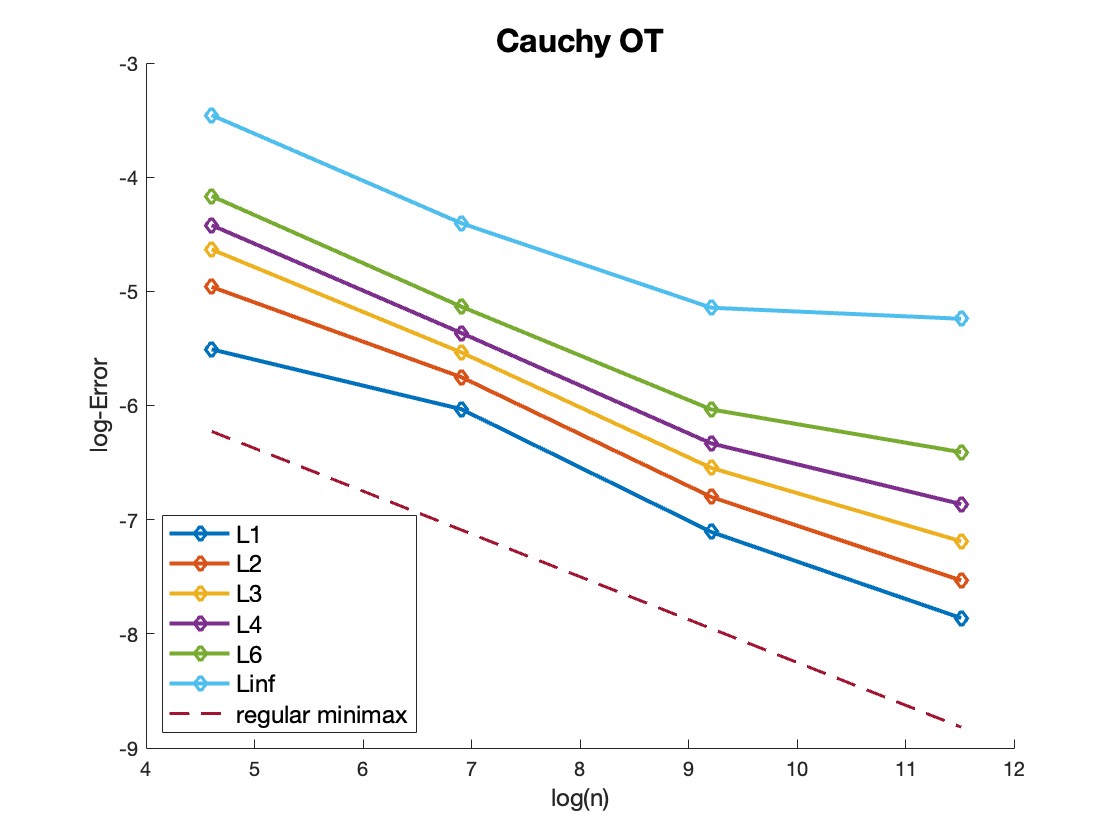}
    \includegraphics[width=0.46\textwidth]{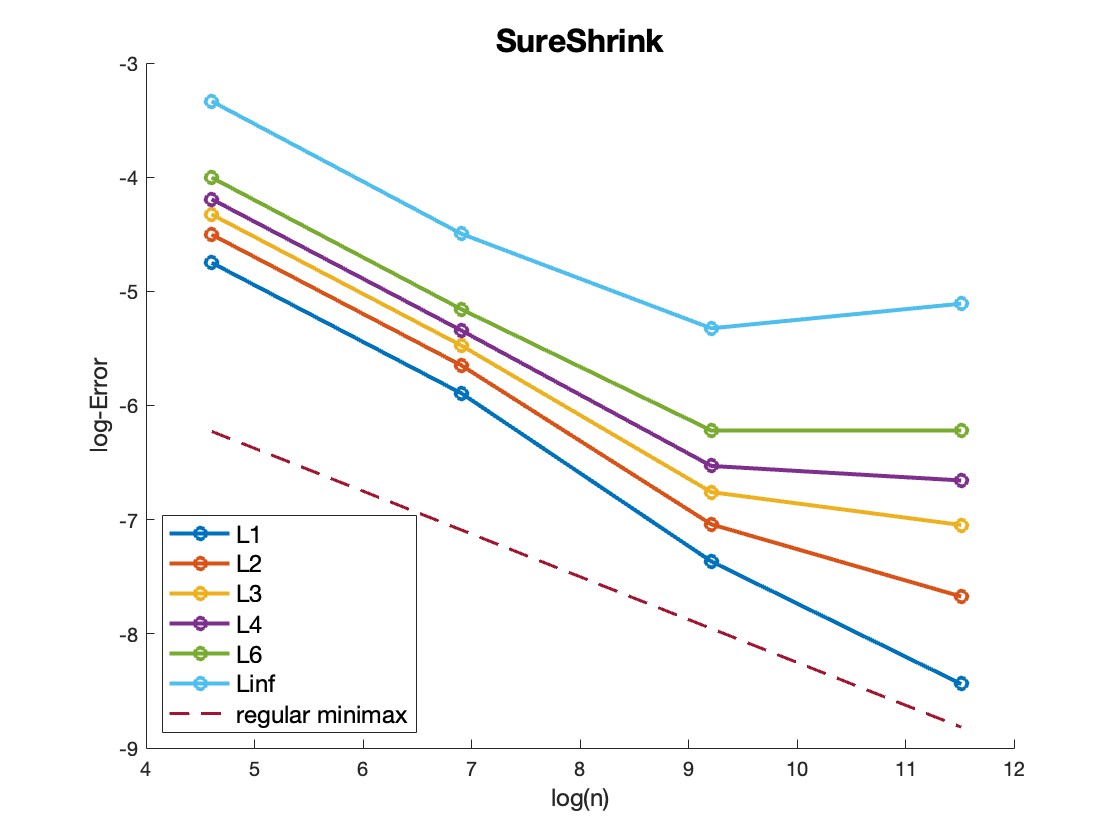}
    \caption{Log average errors for four `least favourable' truths from Figure \ref{fig-badtruths}, with corresponding $n=10^{i+1}$. $L_{p'}$-errors for the posterior mean for Cauchy OT priors (left) and for the hybrid SureShrink estimator (right), for $p'=1,2,3,4,6,\infty$.}
    \label{fig-badtruthserrors}
\end{figure}

\section{Discussion} \label{sec:disc}

In this work we introduce a number of nonparametric priors that are flexible enough to reach (near--) minimax optimality over Besov spaces and over a whole spectrum of loss functions. The prior we recommend is the OT (Oversmoothed heavy-Tailed) prior with scale parameters $\sigma_k=\exp\{-(\log{k})^{1+\nu}\}$, where we recommend the universal choice $\nu=1/2$ yielding uniformly excellent results in practice, although close choices such as $\nu=1$ remain quite close in terms of performance as well; and in terms of type of heavy-tailed prior, taking a Cauchy distribution or a Student($p$) with e.g. $p=3$ degrees of freedom   (where the difference of performance between the two is, again, quite mild). The new lower bounds results of the present paper (Theorem \ref{thm : lbprob}) also confirm that the special type of decrease as above for $\sigma_k$ (slightly faster than any polynomial in $k$) is really needed in order to achieve full adaptation, since we prove that the HT$(\alpha)$ prior only achieves one-sided adaptation in the oversmoothing case $\alpha\ge \beta$, as also seen in our simulation study.

Additional insight can be drawn from the empirical study of the posterior estimators of prototypical spatially inhomogeneous truths under these priors. Indeed, we have seen, for example in the case of the Doppler function (see Figure \ref{fig-DJ94errors} and Table 1 below), that the errors associated with the Cauchy prior become close (if not larger) to those of the Gaussian prior (corresponding to a linear estimator) as $p'$ approaches $\infty$. 
This behavior does not deviate from the theoretical findings in terms of rates, since the linear and global minimax rates are getting closer to each other for larger values of $p'$ (and it is conceivable that the posterior for the Gaussian hierarchical prior with two parameters achieves the linear minimax rate). Interestingly, however, as it was already noted in \cite{AC}, the Cauchy prior still provides better signal denoising when plotting the reconstruction. This suggests that, when comparing two estimates, unless one clearly outperforms the other, examining performance under a single $L_{p'}$ loss is not a sufficient criterion. This is especially the case when $p'$ grows large or equals $\infty$, as linear and global minimax rates then tend to coincide. This suggests that a decision based on the full range of $L_{p'}$ errors can be preferable, even though the question of which loss should ultimately be favored remains open and depends on the specific inferential goals.
 
The present paper also opens the door to the investigation of (very--) heavy tailed priors for nonparametric inference, including the classical Cauchy prior, or the more elaborate horseshoe distribution. The simulations in Section \ref{sec:sims} reveal a number of differences between the OT with Cauchy or Student prior and the horseshoe prior. From the theoretical perspective, the horseshoe prior fits the main Theorem in a similar way as Cauchy. Empirically, if we take a series prior with scalings $\sigma_k$ as in the OT prior with horseshoe distributions,  we obtain roughly similar behaviour as for Cauchy. However, algorithmically, although in the present setting the difference in computing time is not very significant, in general the fact that the horseshoe models a near-spike at zero -- in order to mimic sparse vectors -- is in principle not necessary in the present purely nonparametric context. For more complex models (such as density estimation or classification, not considered here, but studied in terms of certain losses in \cite{AC}), one may think that modelling a near-spike will be more costly computationally. 

Another interesting point of comparison is the {\em truncated} horseshoe prior. Such a prior is generally used in the context of sparsity -- for instance for sparse sequences or high-dimensional linear regression \cite{hs1} --; similarly, a Cauchy prior with common scale parameter was proposed for high dimensional linear regression in \cite{dt12}. While such a prior can be used in the present nonparametrics context as well, it behaves somewhat similarly to a hard-thresholding estimator and so seems less flexible than e.g. the SureShrink algorithm that we compare to in Section \ref{sec:sims}, and indeed its performance is not as favourable for nonparametric function estimation, as one would expect (we did not try to optimise over the parameter $\tau$, which in principle could be done e.g. via empirical Bayes, but the corresponding procedure seems overall more suited to purely sparse classes rather than nonparametric ones). We see that part of the problem with this type of prior is the choice of $\tau$; we expect something similar to occur with spike-and-slab priors as considered in \cite{hrs15}. Indeed, although spike-and-slab priors were shown to lead to powerful posterior contraction (see e.g. \cite{hrs15}) which we conjecture would lead to minimax-optimal sparse Besov--rates, including the optimal logarithmic factors for separable rules, their practical performance typically depends on a careful calibration of the spike weight, which plays a role analogous to the global scale parameter $\tau$ in horseshoe-type priors. By contrast, while the theoretical contraction rates derived for OT priors with scaling \eqref{def : OTprior} exhibit a mild logarithmic overhead (e.g. in Theorem \ref{thm : OTconc}), reflecting intrinsic limitations of such choice of scaling, this does not appear to translate into a loss of efficiency in practice. In numerical experiments, OT priors perform at least on par with known optimal estimators, including soft--thresholding. Also, as discussed in \cite{AC}, deploying spike-and-slab priors in more complex models is expected to lead to computational difficulties in terms of exploration of sets of submodels (i.e. the ones corresponding to where zeros are distributed from the prior/posterior); and, on the other hand, OT priors retain a form of computational tractability beyond regression models, as their built-in deterministic shrinkage (through non-random scales $\sigma_k$) enables to use MCMC without having to sample posteriors on submodels, as would be the case for model-selection type priors such as sieve priors with random cut-off or spike-and-slab priors. 

Let us further comment on the appearance of the additional logarithmic factors in the theoretical upper bounds we derive: these factors are intrinsically linked to the separable nature of the prior and to the specific OT choice of scaling in \eqref{def : OTprior}. One possible way to move beyond the class of separable rules and obtain theoretical results without extra logarithmic factors (at least in $L_2$ loss, as in the setting of Theorem~\ref{thm : OTconc}) would be to follow the interesting idea of \cite{GaoZhouBlock}, who consider priors defined on blocks of grouped coefficients, in the spirit of block-thresholding methods. While this is an interesting direction, we believe that a careful investigation of such priors and their theoretical properties would deserve a separate contribution.

Finally, in the present paper we have considered for simplicity the Gaussian white noise model. We believe that at least part of our results could be extended to nonparametric regression -- with fixed regular design say -- by using again a direct analysis of the posterior using an empirical projection of the observations onto an orthonormal basis; extending the present results, at least those in Section \ref{sec : besov}, to other models using `prior-mass' type arguments (as was done e.g. in \cite{AC} for density estimation and classification, and tempered posteriors) remains an open problem: we refer to \cite{AC}, Appendix A, for more discussion on this. 

\section{Proofs of the results of Section \ref{sec : Sobolev}}
\subsection{Proof of Theorem \ref{thm : OTconc}}\label{proof : OTconc}

We focus here on the case of OT scalings \eqref{def : OTprior} with scale $\nu >0$, the statement involving the truncated horseshoe as in \eqref{def : truncHSprior} is proven in the supplementary material \cite{sm}. Let \[ K_n := \left( \frac{n}{\log^{(1+\kappa)(1+\nu)}n}\right)^{\frac1{2 \beta +1}} \quad \text{and} \quad v_n := \left( \frac{n}{ \log^{(1+\kappa)(1+\nu)}n}\right)^{-\frac{2\beta}{2 \beta +1}} =\frac{K_n}{n}\log^{(1+\kappa)(1+\nu)}n .\] For $f \in L_2$, let $f^{[K_n]}$ denote its orthogonal projection onto the linear span of the first $K_n$ basis vectors (taking the closest integer to $K_n$) and set $f^{[K_n^c]} := f - f^{[K_n]}$. A union bound leads to
\begin{align*}
\post{ \{ f \, : \, ||f - f_0||_2^2 > M_n v_n \}} &\leq \post{ \{ f \, : \, ||f^{[K_n^c]} - f_0^{[K_n^c]}||_2^2 > M_n v_n/2 \}}\\
&+ \post{ \{ f \, : \, ||f^{[K_n]} - f_0^{[K_n]}||_2^2 > M_n v_n/2 \}}.
\end{align*}

Let us first deal with the coefficients $k \leq K_n$. Using Markov's inequality and next splitting the sum with $(a+b)^2 \leq 2a^2 +2b^2$, one can bound the second term in the last display by $2/(M_nv_n)$ times
\[ \int ||f^{[K_n]} - f_0^{[K_n]}||_2^2 \, d \Pi(f \, | \, X) \leq 2  \sum_{k \leq K_n} \int (f_k-X_k)^2 \, d \Pi(f \, | \, X) + 2 \sum_{k \leq K_n} (X_k - f_{0,k})^2 .\]
Under $E_{f_0}$, using the definition of the model \eqref{def : nseq}, the second sum on the right hand side is smaller than $K_n / n = o(v_n)$. It now suffices to show that
\[\sum_{k \leq K_n} E_{f_0} \int (f_k-X_k)^2 \, d \Pi(f \, | \, X)  = o(M_n v_n). \]
Combining Lemma \ref{lem : lap} with $p=2$ and Lemma \ref{lem : lapHTnew} on coordinate $k$ we get, noting that $|f_{0,k}| \leq F$ for all $k$ follows from \eqref{def : sobolev ball}, for any $t \in \R$
\[n E_{f_0} \int (f_k-X_k)^2 \, d \Pi(f \, | \, X) \lesssim t^{-2}  \left[ 1 + \log^{2}(\sigma_k \sqrt{n}) +  t^{4} + \log^{2(1 + \kappa)}\left( 1 + \frac{F + 1/\sqrt{n}}{\sigma_k} \right)\right].\]
Since $\sigma_k = e^{-\log^{1 + \nu}k}$ and $k \leq K_n$, we have $\log\sigma_k^{-1} \lesssim \log^{1+\nu}n $. By taking $t^4 \asymp \log^{2(1+\kappa)(1+\nu)}n$, the bound of the last display is of order at most $\log^{(1+\nu)(1+\kappa)}n $. Therefore, the sum in the last but one display is bounded up to constant, by 
\[\frac{K_n}{n} \log^{(1+\kappa)(1+\nu)}n = v_n = o(M_n v_n).\]
For the terms $k > K_n$, since $f_0 \in \mathcal{S}^{\beta}(F)$, when $n$ is large enough, $ \sum_{k > K_n} |f_{0,k}|^2 \leq v_n/8$, thus
\[ E_{f_0} \post{  \{ f \, : \, \sum_{k > K_n} |f_k - f_{0,k}|^2 > M_n v_n/2 \} } \leq  E_{f_0} \post{  \{ f \, : \, \sum_{k > K_n} |f_k|^2 > M_n v_n/8 \} }. \]
Let us introduce the sequence $z_k := k^{-1} \log^{-2}k$. Using the summability of $(z_k)$ and the divergence of $(M_n)$, as $n$ gets large enough, we have $\sum_{k > K_n} z_k \leq M_n/8$. Thus, as $n$ gets large enough
\begin{align*}
    E_{f_0} \post{  \{ f \, : \, \sum_{k > K_n} |f_k|^2 > M_n v_n/8 \} } \leq E_{f_0} \post{\{ f \, : \, \underset{k > K_n}{\max} z_k^{-1} |f_k|^2  \geq  v_n\}  }.
\end{align*} 
For any event $A_n$, the upper-bound of the last display is further bounded by
\begin{equation}\label{eq : thm1 bound before event tech}
    \sum_{k > K_n} E_{f_0} \left( \post{\{ f \, : \,  f_k^2  \geq  z_k v_n\}  } \mathbf{1}_{A_n} \right) + P_{f_0}(A_n^c) .
\end{equation}
Let us define the event
\begin{equation}\label{def : eventAn}
    A_n := \bigcap_{ K_n < k \leq n, k \in N}A_{k,0} \cap \bigcap_{ \ell \geq 1} \bigcap_{\ell n < k \leq (\ell+1)n , k \in N} A_{k,\ell},
\end{equation}
where we have set \[A_{k,\ell} := \left\{ |X_k| \leq \underbrace{\sqrt{\frac{4 \log (n(\ell + 1)^2)}{n}}}_{y_{k,\ell}}\right\}; \qquad N := \{k \, : \, |f_{0,k}| \leq 1/\sqrt{n} \}.\]
For any $\ell,k$ such that $\ell n < k \leq (\ell+1)n$ we set $x_k = y_{k,\ell} $. With such choice of $(x_k)$ the conditions of Lemma \ref{lem : tech lower bound} are satisfied on the event $A_n$.
Writing $\phi$ for the density of the standard Gaussian distribution, using $\phi \lesssim 1$ for the numerator and for the denominator Lemma \ref{lem : tech lower bound} with $m =0$, on coordinate $k$, we have
\begin{align*}
    E_{f_0} \left( \post{\{ f \, : \,  f_k^2  \geq  z_k v_n\}  } \mathbf{1}_{A_n} \right)  &\leq  E_{f_0} \left(\frac{\int_{|\theta| >  \sqrt{ z_k v_n}} \phi(\sqrt{n}(X_k - \theta)) h (\theta / \sigma_k) \, d \theta}{\int \phi(\sqrt{n}(X_k - \theta)) h (\theta / \sigma_k) \, d \theta} \mathbf{1}_{A_n} \right) \\
    &\lesssim E_{f_0} \left( \frac{1}{\phi(\sqrt{n}x_k)} 2 \overline{H}\left( \sigma_k^{-1}\sqrt{v_n z_k}\right) \mathbf{1}_{A_n} \right).
\end{align*}
Now for any $k > K_n$ and $n$ large enough, we have $\sigma_k^{-1}\sqrt{{v_n}{z_k}} \geq 1$, therefore using assumption (H3)
\[ \overline{H}\left( \sigma_k^{-1}\sqrt{{v_n}{z_k}}\right) \lesssim \frac{\sigma_k}{\sqrt{{v_n}{z_k}}}.\]
Noting that for $y_{k,\ell}$ as defined above  $\phi(\sqrt{n} y_{k,\ell}) \gtrsim (n(\ell + 1)^2)^{-2}$, the sum in \eqref{eq : thm1 bound before event tech} is bounded as
\begin{align*}
   & \sum_{k > K_n} E_{f_0} \left( \post{\{ f \, : \,  f_k^2  \geq  z_k v_n\}  } \mathbf{1}_{A_n} \right)\\
   & \lesssim \sum_{K_n < k \leq n} n^2 \frac{\sigma_k}{\sqrt{{v_n}{z_k}}} + \frac{1}{\sqrt{ v_n}} \sum_{\ell \geq 1} \sum_{\ell n < k \leq (\ell+1)n} (n(\ell+1)^2)^2 \sigma_k z_k^{-1/2} \\
   & \lesssim \frac{n^{7/2} \log n}{\sqrt{v_n}}  e^{-C \log^2n} + \frac{n^{7/2}}{\sqrt{v_n}} \sum_{\ell\geq 1} (\ell+1)^4 \sqrt{\ell+1} \log((\ell+1)n) e^{-\log^2(\ell n)}
\end{align*}
Since $v_n = (n/\log^{(1+\kappa)(1+\nu)}n)^{-2\beta/(2\beta+1)}$, the sums in the last display all go to $0$, when $n \to \infty$. Finally from Lemma \ref{lem : pAn} we have $P_{f_0}(A_n^c) \to 0$, which concludes the proof.

\subsection{Proof of Theorem \ref{thm : lbprob}}\label{proof : lbprob}
Here we only show the proof of the in-probability lower bound, the proof for the in-expectation version can be found in Section C of the supplementary material \cite{sm}. Let $K_{\alpha} := n^{1/(2\alpha +1)}$ and $\varepsilon_n := K_{\alpha}^{-\alpha} \log^{\gamma} n$, where $\gamma > 0$ is to be chosen below. It is enough to show (see for instance \cite{ic08}, in particular Lemma 1 therein) that, as $n \to \infty,$ 
\[\frac{\Pi\left(||f-f_0||_2 \leq \varphi_n\right)}{e^{-2n \varepsilon_n^2}\Pi\left(||f-f_0||_2 \leq \varepsilon_n\right)} \to 0 .\]
Since $\alpha < \beta$, we have $f_0 \in S^{\beta}(F) \subset S^{\alpha}(F)$ and thanks to \cite{AC}, Theorem 6, as soon as $\gamma >0$ is taken large enough, the $HT (\alpha)$ prior satisfies the following prior mass condition, for $n$ large enough,
\begin{equation}\label{eq : prior mass}
    \Pi\left(||f-f_0||_2 \leq \varepsilon_n\right) \geq e^{-n \varepsilon_n^2}.
\end{equation}
For the numerator, looking at the orthogonal projection on coordinates $k > K_{\alpha}$, we have
\[ \Pi\left(||f-f_0||_2 \leq \varphi_n\right) \leq \Pi\left(||f^{[K_{\alpha}^c]}-f_0^{[K_{\alpha}^c]}||_2 \leq \varphi_n\right) .\]
Since $f_0 \in S^{\beta}(F)$ and $\alpha < \beta$, for $n$ large enough, we have $||f_0^{[K_{\alpha}^c]}||_2 \leq F K_{\alpha}^{- \beta} \leq \varphi_n$ by definition of $\varphi_n$. Thus the triangle inequality leads to 
\[ \Pi\left(||f^{[K_{\alpha}^c]}-f_0^{[K_{\alpha}^c]}||_2 \leq \varphi_n\right)  \leq \Pi \left( ||f^{[K_{\alpha}^c]} ||_2 \leq 2\varphi_n \right).\]
To bound the last term, apply Lemma \ref{lem : aurzada} with $\mu := 1/2 + \alpha$, $K_n = K_{\alpha}$ and $\lambda := n (\log n)^{(2 \gamma +1 )(2 \alpha +1)}$ such that $\mu > 1 $ and $K_n \lambda^{-1/(2\mu)} \leq 1$ when $n$ is large enough. Thanks to hypothesis (H3), the condition $E(|\zeta|^{1/\mu}) < \infty$ is satisfied for $\mu >1$ (hence the prior regularity assumption $\alpha>1/2$, see also Remark \ref{rem:lbass}), therefore we obtain
\[ \log \Pi \left( ||f^{[K_{\alpha}^c]} ||_2 \leq 2\varphi_n \right) \leq4 \lambda \varphi_n^2 - C \lambda^{1/(2\alpha +1)} .\]
If $\delta >0 $ is large enough, we have $4\lambda \varphi_n^2 \leq (C/2) \lambda^{1/(2\alpha +1)}$ and the previous bound is smaller than 
\[ - (C/2) n^{1/(2 \alpha +1)} (\log n)^{2 \gamma +1} \lesssim -  n \varepsilon_n^2 \log n. \]
Combining this bound and the upper bound \eqref{eq : prior mass} leads to 
\[ \frac{\Pi\left(||f-f_0||_2 \leq \varphi_n\right)}{e^{-2n \varepsilon_n^2}\Pi\left(||f-f_0||_2 \leq \varepsilon_n\right)} \leq e^{ - C' n \varepsilon_n^2 \log n} \to 0,\]
as $n \to \infty$, which concludes the proof for the in-probability bound.
\subsection{Technical lemmas}

\begin{lemma}\label{lem : lapHTnew}
   Let $\theta \sim \pi$ and $X|\theta \sim \mathcal{N}(\theta, 1/n)$. Suppose $\pi$ is the law of $\sigma \cdot \zeta$ where $\sigma > 0$ and $\zeta$ has an heavy-tailed density satisfying \ref{H1}--\ref{H2}. Then, for all $t \in \R$, $\theta_0 \in \R$, $\sigma > 0$,

    \begin{equation*}
         E_{\theta_0} E^{\pi}\left[e^{t\sqrt{n}(\theta-X)} \, | \, X \right] \lesssim \sigma \sqrt{n }e^{t^2/2} e^{c_1 \log^{1+ \kappa}( 1 + \frac{|\theta_0| + 1 / \sqrt{n}}{\sigma} )} 
    \end{equation*}
\end{lemma}

\begin{proof}
    For any $t \in \R$, denoting $h$ the density of $\zeta$, we have
\begin{align*}
    E_{\theta_0} E^{\pi} \left[e^{t\sqrt{n}(\theta-X)} \, | \, X \right] &= E_{\theta_0} \frac{\int \exp(t \sqrt{n}(\theta - X)) \phi(\sqrt{n}(X - \theta)) h(\theta / \sigma) \, d\theta}{\int \phi(\sqrt{n}(X - \theta)) h(\theta / \sigma) \, d\theta} \\
    &= E_{\xi \sim \mathcal{N}(0,1)} \frac{\int e^{t(v-\xi) - \frac{(v-\xi)^2}{2}} h(\frac{\theta_0 + v / \sqrt{n}}{\sigma}) \, dv}{\int e^{ - \frac{(v-\xi)^2}{2}} h(\frac{\theta_0 + v / \sqrt{n}}{\sigma}) \, dv}.
\end{align*}
The integral on the numerator is bounded as follows, using that $h$ is a density function
\begin{align*}
    \int e^{t(v-\xi) - \frac{(v-\xi)^2}{2}} h(\frac{\theta_0 + v / \sqrt{n}}{\sigma}) \, dv &= e^{t^2/2} \int e^{- \frac{(t-(v-\xi))^2}{2}} h(\frac{\theta_0 + v / \sqrt{n}}{\sigma}) \, dv 
     \leq \sigma \sqrt{n} e^{t^2/2}
\end{align*}
For the denominator using that $h$ is symmetric and decreasing on $(0,\infty)$, along with condition \ref{H2}, one gets

\begin{align*}
    \int e^{ - \frac{(v-\xi)^2}{2}} h(\frac{\theta_0 + v / \sqrt{n}}{\sigma}) \, dv &\gtrsim  \int_{-1}^1 e^{-\frac{(v-\xi)^2}{2}} e^{-c_1 \log^{1+ \kappa}( 1 + \frac{\theta_0 + v / \sqrt{n}}{\sigma} )} \, dv \\
    & \gtrsim e^{-c_1 \log^{1+ \kappa}( 1 + \frac{|\theta_0| + 1 / \sqrt{n}}{\sigma} )} \int_{-1}^1 e^{-\frac{(v-\xi)^2}{2}} \, dv.
\end{align*}
Using \cite{CN} pages 2015-2016, there exists a universal constant $ K >0$, such that 
\[ E_{\xi \sim \mathcal{N}(0,1)} \left[ \left( \int_{-1}^1 e^{-\frac{(v-\xi)^2}{2}} \, dv. \right)^{-1} \right] \leq K.\]
Combining the previous inequalities leads to the desired result.
\end{proof}

\begin{lemma}\label{lem : lap}
    Let $Y$ be a real valued random variable. Then for $t > 0$, $p \geq 1$ and $\mathcal{L}(t) := E(\exp(t|Y|)),$
    \begin{equation*}
        E(|Y|^p) \leq \frac{2^{p-1}}{t^p} \left(p^p + \log^p \mathcal{L}(t)\right)
    \end{equation*}
\end{lemma}
\begin{proof}
    Write $E(|Y|^p) = t^{-p} E[\log^p(\exp(t|Y|))]$. Using concavity of the map $x \mapsto \log^p x$ for $x > e^{p-1},$ by Jensen's inequality, one may write
    \begin{align*}
        E[\log^p(\exp(t|Y|))] & \leq E[\log^p(e^{p-1} + \exp(t|Y|))] \leq \log^p(e^{p-1} + \mathcal{L}(t)) \\
        &\leq (p + \log \mathcal{L}(t))^p \leq 2^{p-1}(p^p + \log^p \mathcal{L}(t)),
    \end{align*}
    where the second to last inequality uses $\log(e^{p-1} + x) \leq p + \log x$ valid for $x \geq 1$.
\end{proof}

\begin{lemma}[Lower bounds for posterior integrals]\label{lem : tech lower bound}
    Let $\phi$ be the density function of the standard Gaussian distribution. Let $\sigma > 0$,  $m \geq 0$ and $h$ be a density function satisfying \ref{H1}. Let $X$ be a random variable and $x$ a deterministic non-negative real number such that $|X| \leq x$ and $\sigma \lesssim x$, we have
    \[ \int |\theta|^m \phi(\sqrt{n}(X - \theta)) h(\theta/\sigma) \, d\theta \gtrsim \sigma^{m+1} \phi(\sqrt{n} x ).\]
\end{lemma}
\begin{proof} 
    By symmetry of both $h$ and $\phi$, it is enough to focus on the case $X \geq 0$. By restricting the domain of integration to the set $[X - x, X + x],$ the integral of interest is greater than
\begin{equation*}
     \phi(\sqrt{n}x) \int_{X - x}^{X + x} |\theta|^m h(\theta / \sigma) \, d\theta.
\end{equation*}
By assumption, $X \leq x$, the integral in the last display can be further bounded from below by $\int_0^{x} |\theta|^m h(\theta / \sigma) \, d\theta = \sigma^{m+1} \int_0^{x/\sigma} |u|^m h(u) \, du,$ recalling $X \geq 0$. Using $\sigma \lesssim x $ as well as the positivity and monotonicity of $h$, the later integral is further bounded below, for some $c > 0$, by $\int_0^c |u|^m h(u) \, du \gtrsim 1$. Putting everything together and using symmetry, one gets the desired result.
\end{proof}

\begin{lemma}\label{lem : pAn}
   Let $A_n$ be the event defined in \eqref{def : eventAn} and assume $f_0 \in \mathcal{S}^{\beta}(F)$, with $\beta,F>0$. Then as $n \to \infty$, we have
    \[ P_{f_0}(A_n^c) \to 0.\]
\end{lemma}
\begin{proof}
For any $k \in N$, by definition we have $|f_{0,k}| \leq 1/\sqrt{n}$, thus for any $\ell \geq 0$,
    \begin{align*}
        P_{f_0}(A_{k,\ell}^c) &= P_{f_0} \left\{ \left|f_{0,k} + \frac{\xi_k}{\sqrt{n}}\right| >\sqrt{\frac{4 \log (n(\ell+1)^2)}{n}} \right\} \leq P\left\{ \left| \mathcal{N}(0,1) \right| >\sqrt{4 \log (n(\ell+1)^2)}-1\right\}.
    \end{align*}
    As $n$ gets large enough, this Gaussian tail probability is further upper-bounded, for any $\ell \geq 0$, by 
    \[P\left\{ \left| \mathcal{N}(0,1) \right| >\sqrt{3 \log (n(\ell+1)^2)}\right\} \leq (n(\ell+1)^2)^{-3/2}.\]
    Combined with \eqref{def : eventAn} the definition of the event $A_n$ and a union bound, we obtain
    \[ P_{f_0}(A_n^c)\leq \sum_{\ell \geq 0}(\ell +1)n \times(n(\ell+1)^2)^{-3/2} \leq n^{-1/2} \sum_{\ell \geq 0} (\ell +1)^{-2} \lesssim 1/\sqrt{n},\]
     which ensures $P_{f_0}(A_n^c) \to 0$ as $n \to \infty$.
\end{proof}
\begin{lemma}\label{lem : aurzada}
Consider a random sum of the form 
\[S_n:=\sum_{k>K_n}\sigma_k^2\zeta_k^2,\]
where $\sigma_k=k^{-\mu
}$ for some $\mu>1/2$ and $\zeta_k$ are independent and identically distributed copies of the real random variable $\zeta$, satisfying $E[|\zeta|^{1/\mu}]<\infty$. Then for any $\varepsilon>0$, it holds 
\begin{equation}\label{eq:sb}
\log P\big(S_n\le \varepsilon^2\big)\leq \lambda \varepsilon^2-C\lambda^{\frac{1}{2\mu}},
\end{equation}
for any $\lambda>0$ such that $K_n\lambda^{-\frac{1}{2\mu}}\leq1$, where $C$ is a positive constant (depending on $\mu$ and the distribution of $\zeta$).
\end{lemma}
\begin{proof}
The proof follows the ideas of \cite{FA06}. First, notice that for any $\lambda>0$, by the (exponential) Markov inequality it holds
\[\log P(S_n\leq \varepsilon^2)\leq \lambda \varepsilon^2+\log Ee^{-\lambda S_n}.\]
By independence, the second term is equal to
\[\log \prod_{k>K_n}Ee^{-\lambda\sigma_k^2\zeta_k^2}=\sum_{k>K_n}\log Ee^{-\lambda\sigma_k^2\zeta_k^2}\leq \int_{K_n}^\infty \log Ee^{-\lambda x^{-2\mu}\zeta^2}dx,\]
 where in the last bound we have used comparison of sum with an integral (taking advantage of the fact that the expectation in the integrand is increasing with $x$). Changing variables twice, first setting $\sqrt{\lambda}x^{-\mu}=y$ and then $y=z^{-\mu}$, we get that the last integral is equal to 
\[\frac{1}{\mu}\lambda^{\frac{1}{2\mu}}\int_0^{\sqrt{\lambda}K_n^{-\mu}}\big(\log E^{-y^2\zeta^2}\big)y^{-\frac1\mu-1}dy=\lambda^{\frac{1}{2\mu}}\int_{K_n\lambda^{-\frac1{2\mu}}}^\infty\log Ee^{-z^{-2\mu}\zeta^2}dz.\]
Noting that the integrand in the right hand side is non-positive, under the assumption $K_n\lambda^{-\frac1{2\mu}}\leq1$, we have that the right hand side is upper bounded by
\[\lambda^{\frac1{2\mu}}\int_1^\infty \log Ee^{-z^{-2\mu}\zeta^2}dz,\]
where \cite[Lemma 4.3]{FA06} shows that the last integral is finite if and only if $E[|\zeta|^{1/\mu}]<\infty$, as assumed. The claimed bound \eqref{eq:sb} follows by combining the above considerations.
\end{proof}

 \noindent{\Large \textbf{Funding}}
 \medskip
 
 \noindent
    IC acknowledges funding from ANR grant project BACKUP ANR-23-CE40-0018-01

\bibliographystyle{imsart-number}
\bibliography{ace}

\begin{thebibliography}{51}

\bibitem{AC}
\begin{barticle}[author]
\bauthor{\bsnm{Agapiou},~\bfnm{Sergios}\binits{S.}} \AND \bauthor{\bsnm{Castillo},~\bfnm{Isma\"el}\binits{I.}}
(\byear{2024}).
\btitle{Heavy-tailed {B}ayesian nonparametric adaptation}.
\bjournal{Ann. Statist.}
\bvolume{52}
\bpages{1433--1459}.
\bdoi{10.1214/24-aos2397}
\bmrnumber{4804815}
\end{barticle}
\endbibitem

\bibitem{sm}
\begin{barticle}[author]
\bauthor{\bsnm{Agapiou},~\bfnm{S.}\binits{S.}}, \bauthor{\bsnm{Castillo},~\bfnm{I.}\binits{I.}} \AND \bauthor{\bsnm{Egels},~\bfnm{Paul}\binits{P.}}
\btitle{Supplement to "Heavy-tailed and Horseshoe priors for regression and sparse Besov rates"}.
\end{barticle}
\endbibitem

\bibitem{adh21}
\begin{barticle}[author]
\bauthor{\bsnm{Agapiou},~\bfnm{Sergios}\binits{S.}}, \bauthor{\bsnm{Dashti},~\bfnm{Masoumeh}\binits{M.}} \AND \bauthor{\bsnm{Helin},~\bfnm{Tapio}\binits{T.}}
(\byear{2021}).
\btitle{Rates of contraction of posterior distributions based on {$p$}-exponential priors}.
\bjournal{Bernoulli}
\bvolume{27}
\bpages{1616--1642}.
\bdoi{10.3150/20-bej1285}
\bmrnumber{4278794}
\end{barticle}
\endbibitem

\bibitem{as22}
\begin{barticle}[author]
\bauthor{\bsnm{Agapiou},~\bfnm{Sergios}\binits{S.}} \AND \bauthor{\bsnm{Savva},~\bfnm{Aimilia}\binits{A.}}
(\byear{2024}).
\btitle{Adaptive inference over {B}esov spaces in the white noise model using {$p$}-exponential priors}.
\bjournal{Bernoulli}
\bvolume{30}
\bpages{2275--2300}.
\bdoi{10.3150/23-bej1673}
\bmrnumber{4746608}
\end{barticle}
\endbibitem

\bibitem{aw21}
\begin{barticle}[author]
\bauthor{\bsnm{Agapiou},~\bfnm{S.}\binits{S.}} \AND \bauthor{\bsnm{Wang},~\bfnm{S.}\binits{S.}}
(\byear{2024}).
\btitle{Laplace priors and spatial inhomogeneity in {B}ayesian inverse problems}.
\bjournal{Bernoulli}
\bvolume{30}
\bpages{878--910}.
\bdoi{10.3150/22-bej1563}
\bmrnumber{4699538}
\end{barticle}
\endbibitem

\bibitem{arbel13}
\begin{barticle}[author]
\bauthor{\bsnm{Arbel},~\bfnm{J.}\binits{J.}}, \bauthor{\bsnm{Gayraud},~\bfnm{G.}\binits{G.}} \AND \bauthor{\bsnm{Rousseau},~\bfnm{J.}\binits{J.}}
(\byear{2013}).
\btitle{Bayesian Optimal Adaptive Estimation Using a Sieve Prior}.
\bjournal{Scandinavian Journal of Statistics}
\bvolume{40}
\bpages{549--570}.
\end{barticle}
\endbibitem

\bibitem{FA06}
\begin{barticle}[author]
\bauthor{\bsnm{Aurzada},~\bfnm{Frank}\binits{F.}}
(\byear{2007}).
\btitle{On the lower tail probabilities of some random sequences in {$l_p$}}.
\bjournal{Journal of Theoretical Probability}
\bvolume{20}
\bpages{843--858}.
\end{barticle}
\endbibitem

\bibitem{cai08}
\begin{barticle}[author]
\bauthor{\bsnm{Cai},~\bfnm{T.~T.}\binits{T.~T.}}
(\byear{2008}).
\btitle{On information pooling, adaptability and superefficiency in nonparametric function estimation}.
\bjournal{Journal of Multivariate Analysis}
\bvolume{99}
\bpages{421--436}.
\end{barticle}
\endbibitem

\bibitem{hs1}
\begin{barticle}[author]
\bauthor{\bsnm{Carvalho},~\bfnm{Carlos~M.}\binits{C.~M.}}, \bauthor{\bsnm{Polson},~\bfnm{Nicholas~G.}\binits{N.~G.}} \AND \bauthor{\bsnm{Scott},~\bfnm{James~G.}\binits{J.~G.}}
(\byear{2010}).
\btitle{The horseshoe estimator for sparse signals}.
\bjournal{Biometrika}
\bvolume{97}
\bpages{465--480}.
\bdoi{10.1093/biomet/asq017}
\bmrnumber{2650751}
\end{barticle}
\endbibitem

\bibitem{ic08}
\begin{barticle}[author]
\bauthor{\bsnm{Castillo},~\bfnm{I.}\binits{I.}}
(\byear{2008}).
\btitle{Lower bounds for posterior rates with {G}aussian process priors}.
\bjournal{Electron. J. Stat.}
\bvolume{2}
\bpages{1281--1299}.
\bdoi{10.1214/08-EJS273}
\bmrnumber{2471287}
\end{barticle}
\endbibitem

\bibitem{ic14}
\begin{barticle}[author]
\bauthor{\bsnm{Castillo},~\bfnm{I.}\binits{I.}}
(\byear{2014}).
\btitle{On {B}ayesian supremum norm contraction rates}.
\bjournal{Ann. Statist.}
\bvolume{42}
\bpages{2058--2091}.
\bdoi{10.1214/14-AOS1253}
\bmrnumber{3262477}
\end{barticle}
\endbibitem

\bibitem{icsfl}
\begin{bbook}[author]
\bauthor{\bsnm{Castillo},~\bfnm{Isma\"el}\binits{I.}}
(\byear{2024}).
\btitle{Bayesian nonparametric statistics, St Flour lecture notes}
\bvolume{LI}.
\bpublisher{Springer}.
\end{bbook}
\endbibitem

\bibitem{ce24}
\begin{barticle}[author]
\bauthor{\bsnm{Castillo},~\bfnm{Isma{{\"e}}l}\binits{I.}} \AND \bauthor{\bsnm{Egels},~\bfnm{Paul}\binits{P.}}
(\byear{2025}).
\btitle{Posterior and Variational Inference for Deep Neural Networks with Heavy-Tailed Weights}.
\bjournal{Journal of Machine Learning Research}
\bvolume{26}
\bpages{1--58}.
\end{barticle}
\endbibitem

\bibitem{cm21}
\begin{barticle}[author]
\bauthor{\bsnm{Castillo},~\bfnm{Isma\"{e}l}\binits{I.}} \AND \bauthor{\bsnm{Mismer},~\bfnm{Romain}\binits{R.}}
(\byear{2021}).
\btitle{Spike and slab {P}\'{o}lya tree posterior densities: adaptive inference}.
\bjournal{Ann. Inst. Henri Poincar\'{e} Probab. Stat.}
\bvolume{57}
\bpages{1521--1548}.
\bdoi{10.1214/20-aihp1132}
\bmrnumber{4291462}
\end{barticle}
\endbibitem

\bibitem{CN}
\begin{barticle}[author]
\bauthor{\bsnm{Castillo},~\bfnm{Isma{\"e}l}\binits{I.}} \AND \bauthor{\bsnm{Nickl},~\bfnm{Richard}\binits{R.}}
(\byear{2013}).
\btitle{{Nonparametric Bernstein–von Mises theorems in Gaussian white noise}}.
\bjournal{The Annals of Statistics}
\bvolume{41}
\bpages{1999 -- 2028}.
\bdoi{10.1214/13-AOS1133}
\end{barticle}
\endbibitem

\bibitem{cr22}
\begin{barticle}[author]
\bauthor{\bsnm{Castillo},~\bfnm{Isma\"el}\binits{I.}} \AND \bauthor{\bsnm{Randrianarisoa},~\bfnm{Thibault}\binits{T.}}
(\byear{2022}).
\btitle{Optional {P}\'olya trees: posterior rates and uncertainty quantification}.
\bjournal{Electron. J. Stat.}
\bvolume{16}
\bpages{6267--6312}.
\bdoi{10.1214/22-ejs2086}
\bmrnumber{4515717}
\end{barticle}
\endbibitem

\bibitem{cr21}
\begin{barticle}[author]
\bauthor{\bsnm{Castillo},~\bfnm{Isma\"el}\binits{I.}} \AND \bauthor{\bsnm{Ro\v{c}kov\'{a}},~\bfnm{Veronika}\binits{V.}}
(\byear{2021}).
\btitle{Uncertainty quantification for {B}ayesian {CART}}.
\bjournal{Ann. Statist.}
\bvolume{49}
\bpages{3482--3509}.
\bdoi{10.1214/21-aos2093}
\bmrnumber{4352538}
\end{barticle}
\endbibitem

\bibitem{cohen}
\begin{barticle}[author]
\bauthor{\bsnm{Cohen},~\bfnm{Albert}\binits{A.}}, \bauthor{\bsnm{Daubechies},~\bfnm{Ingrid}\binits{I.}} \AND \bauthor{\bsnm{Vial},~\bfnm{Pierre}\binits{P.}}
(\byear{1993}).
\btitle{Wavelets on the Interval and Fast Wavelet Transforms}.
\bjournal{Applied and Computational Harmonic Analysis}
\bvolume{1}
\bpages{54-81}.
\bdoi{https://doi.org/10.1006/acha.1993.1005}
\end{barticle}
\endbibitem

\bibitem{dt12}
\begin{barticle}[author]
\bauthor{\bsnm{Dalalyan},~\bfnm{A.~S.}\binits{A.~S.}} \AND \bauthor{\bsnm{Tsybakov},~\bfnm{A.~B.}\binits{A.~B.}}
(\byear{2012}).
\btitle{Sparse regression learning by aggregation and {L}angevin {M}onte-{C}arlo}.
\bjournal{J. Comput. System Sci.}
\bvolume{78}
\bpages{1423--1443}.
\bdoi{10.1016/j.jcss.2011.12.023}
\bmrnumber{2926142}
\end{barticle}
\endbibitem

\bibitem{daubechiesbook}
\begin{bbook}[author]
\bauthor{\bsnm{Daubechies},~\bfnm{I.}\binits{I.}}
(\byear{1992}).
\btitle{Ten lectures on wavelets}.
\bseries{CBMS-NSF Regional Conference Series in Applied Mathematics}
\bvolume{61}.
\bpublisher{Society for Industrial and Applied Mathematics (SIAM), Philadelphia, PA}.
\bdoi{10.1137/1.9781611970104}
\bmrnumber{1162107}
\end{bbook}
\endbibitem

\bibitem{doleraetal24}
\begin{barticle}[author]
\bauthor{\bsnm{Dolera},~\bfnm{Emanuele}\binits{E.}}, \bauthor{\bsnm{Favaro},~\bfnm{Stefano}\binits{S.}} \AND \bauthor{\bsnm{Giordano},~\bfnm{Matteo}\binits{M.}}
(\byear{2024}).
\btitle{On strong posterior contraction rates for Besov-Laplace priors in the white noise model}.
\bnote{arXiv preprint arxiv:2411.06981}.
\end{barticle}
\endbibitem

\bibitem{wavelab}
\begin{barticle}[author]
\bauthor{\bsnm{Donoho},~\bfnm{David}\binits{D.}}, \bauthor{\bsnm{Maleki},~\bfnm{Arian}\binits{A.}}, \bauthor{\bsnm{Shahram},~\bfnm{Morteza}\binits{M.}} \betal{et~al.}
(\byear{2006}).
\btitle{Wavelab 850}.
\bjournal{Software toolkit for time-frequency analysis}.
\end{barticle}
\endbibitem

\bibitem{DJ94}
\begin{barticle}[author]
\bauthor{\bsnm{Donoho},~\bfnm{David~L.}\binits{D.~L.}} \AND \bauthor{\bsnm{Johnstone},~\bfnm{Iain~M.}\binits{I.~M.}}
(\byear{1994}).
\btitle{Ideal spatial adaptation by wavelet shrinkage}.
\bjournal{Biometrika}
\bvolume{81}
\bpages{425--455}.
\bdoi{10.1093/biomet/81.3.425}
\bmrnumber{1311089}
\end{barticle}
\endbibitem

\bibitem{DJ95}
\begin{barticle}[author]
\bauthor{\bsnm{Donoho},~\bfnm{David~L.}\binits{D.~L.}} \AND \bauthor{\bsnm{Johnstone},~\bfnm{Iain~M.}\binits{I.~M.}}
(\byear{1995}).
\btitle{Adapting to unknown smoothness via wavelet shrinkage}.
\bjournal{J. Amer. Statist. Assoc.}
\bvolume{90}
\bpages{1200--1224}.
\bmrnumber{1379464}
\end{barticle}
\endbibitem

\bibitem{djkp95}
\begin{barticle}[author]
\bauthor{\bsnm{Donoho},~\bfnm{David~L.}\binits{D.~L.}}, \bauthor{\bsnm{Johnstone},~\bfnm{Iain~M.}\binits{I.~M.}}, \bauthor{\bsnm{Kerkyacharian},~\bfnm{G\'erard}\binits{G.}} \AND \bauthor{\bsnm{Picard},~\bfnm{Dominique}\binits{D.}}
(\byear{1995}).
\btitle{Wavelet shrinkage: asymptopia?}
\bjournal{J. Roy. Statist. Soc. Ser. B}
\bvolume{57}
\bpages{301--369}.
\bnote{With discussion and a reply by the authors}.
\bmrnumber{1323344}
\end{barticle}
\endbibitem

\bibitem{djkp96}
\begin{barticle}[author]
\bauthor{\bsnm{Donoho},~\bfnm{David~L.}\binits{D.~L.}}, \bauthor{\bsnm{Johnstone},~\bfnm{Iain~M.}\binits{I.~M.}}, \bauthor{\bsnm{Kerkyacharian},~\bfnm{G\'erard}\binits{G.}} \AND \bauthor{\bsnm{Picard},~\bfnm{Dominique}\binits{D.}}
(\byear{1996}).
\btitle{Density estimation by wavelet thresholding}.
\bjournal{Ann. Statist.}
\bvolume{24}
\bpages{508--539}.
\bdoi{10.1214/aos/1032894451}
\bmrnumber{1394974}
\end{barticle}
\endbibitem

\bibitem{djkp97}
\begin{bincollection}[author]
\bauthor{\bsnm{Donoho},~\bfnm{D.~L.}\binits{D.~L.}}, \bauthor{\bsnm{Johnstone},~\bfnm{I.~M.}\binits{I.~M.}}, \bauthor{\bsnm{Kerkyacharian},~\bfnm{G.}\binits{G.}} \AND \bauthor{\bsnm{Picard},~\bfnm{D.}\binits{D.}}
(\byear{1997}).
\btitle{Universal near minimaxity of wavelet shrinkage}.
In \bbooktitle{Festschrift for {L}ucien {L}e {C}am}
\bpages{183--218}.
\bpublisher{Springer, New York}.
\bdoi{10.1007/978-1-4612-1880-7\_12}
\bmrnumber{1462946}
\end{bincollection}
\endbibitem

\bibitem{GaoZhouBlock}
\begin{barticle}[author]
\bauthor{\bsnm{Gao},~\bfnm{Chao}\binits{C.}} \AND \bauthor{\bsnm{Zhou},~\bfnm{Harrison~H.}\binits{H.~H.}}
(\byear{2016}).
\btitle{{Rate exact Bayesian adaptation with modified block priors}}.
\bjournal{The Annals of Statistics}
\bvolume{44}
\bpages{318 -- 345}.
\bdoi{10.1214/15-AOS1368}
\end{barticle}
\endbibitem

\bibitem{ggv00}
\begin{barticle}[author]
\bauthor{\bsnm{Ghosal},~\bfnm{Subhashis}\binits{S.}}, \bauthor{\bsnm{Ghosh},~\bfnm{Jayanta~K.}\binits{J.~K.}} \AND \bauthor{\bparticle{van~der} \bsnm{Vaart},~\bfnm{Aad~W.}\binits{A.~W.}}
(\byear{2000}).
\btitle{Convergence rates of posterior distributions}.
\bjournal{Ann. Statist.}
\bvolume{28}
\bpages{500--531}.
\bmrnumber{MR1790007 (2001m:62065)}
\end{barticle}
\endbibitem

\bibitem{gvbook}
\begin{bbook}[author]
\bauthor{\bsnm{Ghosal},~\bfnm{Subhashis}\binits{S.}} \AND \bauthor{\bparticle{van~der} \bsnm{Vaart},~\bfnm{Aad}\binits{A.}}
(\byear{2017}).
\btitle{Fundamentals of nonparametric {B}ayesian inference}.
\bseries{Cambridge Series in Statistical and Probabilistic Mathematics}
\bvolume{44}.
\bpublisher{Cambridge University Press, Cambridge}.
\bdoi{10.1017/9781139029834}
\bmrnumber{3587782}
\end{bbook}
\endbibitem

\bibitem{gn11}
\begin{barticle}[author]
\bauthor{\bsnm{Gin\'{e}},~\bfnm{Evarist}\binits{E.}} \AND \bauthor{\bsnm{Nickl},~\bfnm{Richard}\binits{R.}}
(\byear{2011}).
\btitle{Rates of contraction for posterior distributions in {$L^r$}-metrics, {$1\leq r\leq\infty$}}.
\bjournal{Ann. Statist.}
\bvolume{39}
\bpages{2883--2911}.
\bdoi{10.1214/11-AOS924}
\bmrnumber{3012395}
\end{barticle}
\endbibitem

\bibitem{GineNickl}
\begin{bbook}[author]
\bauthor{\bsnm{Gin{\'e}},~\bfnm{Evarist}\binits{E.}} \AND \bauthor{\bsnm{Nickl},~\bfnm{Richard}\binits{R.}}
(\byear{2015}).
\btitle{Mathematical foundations of infinite-dimensional statistical models}
\bvolume{40}.
\bpublisher{Cambridge university press}.
\end{bbook}
\endbibitem

\bibitem{g23}
\begin{barticle}[author]
\bauthor{\bsnm{Giordano},~\bfnm{Matteo}\binits{M.}}
(\byear{2023}).
\btitle{Besov-{L}aplace priors in density estimation: optimal posterior contraction rates and adaptation}.
\bjournal{Electron. J. Stat.}
\bvolume{17}
\bpages{2210--2249}.
\bdoi{10.1214/23-ejs2161}
\bmrnumber{4649387}
\end{barticle}
\endbibitem

\bibitem{hkpt98}
\begin{bbook}[author]
\bauthor{\bsnm{H\"ardle},~\bfnm{Wolfgang}\binits{W.}}, \bauthor{\bsnm{Kerkyacharian},~\bfnm{Gerard}\binits{G.}}, \bauthor{\bsnm{Picard},~\bfnm{Dominique}\binits{D.}} \AND \bauthor{\bsnm{Tsybakov},~\bfnm{Alexander}\binits{A.}}
(\byear{1998}).
\btitle{Wavelets, approximation, and statistical applications}.
\bseries{Lecture Notes in Statistics}
\bvolume{129}.
\bpublisher{Springer-Verlag, New York}.
\bdoi{10.1007/978-1-4612-2222-4}
\bmrnumber{1618204}
\end{bbook}
\endbibitem

\bibitem{hrs15}
\begin{barticle}[author]
\bauthor{\bsnm{Hoffmann},~\bfnm{Marc}\binits{M.}}, \bauthor{\bsnm{Rousseau},~\bfnm{Judith}\binits{J.}} \AND \bauthor{\bsnm{Schmidt-Hieber},~\bfnm{Johannes}\binits{J.}}
(\byear{2015}).
\btitle{On adaptive posterior concentration rates}.
\bjournal{Ann. Statist.}
\bvolume{43}
\bpages{2259--2295}.
\bdoi{10.1214/15-AOS1341}
\bmrnumber{3396985}
\end{barticle}
\endbibitem

\bibitem{lmr25}
\begin{barticle}[author]
\bauthor{\bsnm{Lacour},~\bfnm{Claire}\binits{C.}}, \bauthor{\bsnm{Massart},~\bfnm{Pascal}\binits{P.}} \AND \bauthor{\bsnm{Rivoirard},~\bfnm{Vincent}\binits{V.}}
(\byear{2025}).
\btitle{Is model selection possible for the $\ell_p $-loss? PCO estimation for regression models}.
\bjournal{arXiv preprint arXiv:2504.11217}.
\end{barticle}
\endbibitem

\bibitem{Lepski}
\begin{barticle}[author]
\bauthor{\bsnm{Lepski},~\bfnm{Oleg}\binits{O.}}
(\byear{2015}).
\btitle{Adaptive estimation over anisotropic functional classes via oracle approach}.
\bjournal{Ann. Statist.}
\bvolume{43}
\bpages{1178--1242}.
\bdoi{10.1214/14-AOS1306}
\bmrnumber{3346701}
\end{barticle}
\endbibitem

\bibitem{lepmamspok97}
\begin{barticle}[author]
\bauthor{\bsnm{Lepski},~\bfnm{O.~V.}\binits{O.~V.}}, \bauthor{\bsnm{Mammen},~\bfnm{E.}\binits{E.}} \AND \bauthor{\bsnm{Spokoiny},~\bfnm{V.~G.}\binits{V.~G.}}
(\byear{1997}).
\btitle{Optimal spatial adaptation to inhomogeneous smoothness: an approach based on kernel estimates with variable bandwidth selectors}.
\bjournal{Ann. Statist.}
\bvolume{25}
\bpages{929--947}.
\bdoi{10.1214/aos/1069362731}
\bmrnumber{1447734}
\end{barticle}
\endbibitem

\bibitem{naulet22}
\begin{barticle}[author]
\bauthor{\bsnm{Naulet},~\bfnm{Zacharie}\binits{Z.}}
(\byear{2022}).
\btitle{Adaptive {B}ayesian density estimation in sup-norm}.
\bjournal{Bernoulli}
\bvolume{28}
\bpages{1284--1308}.
\bdoi{10.3150/21-bej1387}
\bmrnumber{4388939}
\end{barticle}
\endbibitem

\bibitem{npt83}
\begin{barticle}[author]
\bauthor{\bsnm{Nemirovski\u{i}},~\bfnm{A.~S.}\binits{A.~S.}}, \bauthor{\bsnm{Polyak},~\bfnm{B.~T.}\binits{B.~T.}} \AND \bauthor{\bsnm{Tsybakov},~\bfnm{A.~B.}\binits{A.~B.}}
(\byear{1983}).
\btitle{Estimates of the maximum likelihood type for a nonparametric regression}.
\bjournal{Dokl. Akad. Nauk SSSR}
\bvolume{273}
\bpages{1310--1314}.
\bmrnumber{731296}
\end{barticle}
\endbibitem

\bibitem{npt85}
\begin{barticle}[author]
\bauthor{\bsnm{Nemirovski\u{i}},~\bfnm{A.~S.}\binits{A.~S.}}, \bauthor{\bsnm{Polyak},~\bfnm{B.~T.}\binits{B.~T.}} \AND \bauthor{\bsnm{Tsybakov},~\bfnm{A.~B.}\binits{A.~B.}}
(\byear{1985}).
\btitle{The rate of convergence of nonparametric estimates of maximum likelihood type}.
\bjournal{Problemy Peredachi Informatsii}
\bvolume{21}
\bpages{17--33}.
\bmrnumber{820705}
\end{barticle}
\endbibitem

\bibitem{nem85}
\begin{barticle}[author]
\bauthor{\bsnm{Nemirovskiy},~\bfnm{A.~S.}\binits{A.~S.}}
(\byear{1985}).
\btitle{Nonparametric estimation of smooth regression functions}.
\bjournal{Soviet J. Comput. Systems Sci.}
\bvolume{23}
\bpages{1--11}.
\bmrnumber{844292}
\end{barticle}
\endbibitem

\bibitem{hs2}
\begin{bincollection}[author]
\bauthor{\bsnm{Polson},~\bfnm{Nicholas~G.}\binits{N.~G.}} \AND \bauthor{\bsnm{Scott},~\bfnm{James~G.}\binits{J.~G.}}
(\byear{2011}).
\btitle{Shrink globally, act locally: sparse {B}ayesian regularization and prediction}.
In \bbooktitle{Bayesian statistics 9}
\bpages{501--538}.
\bpublisher{Oxford Univ. Press, Oxford}
\bnote{With discussions by Bertrand Clark, C. Severinski, Merlise A. Clyde, Robert L. Wolpert, Jim e. Griffin, Philiip J. Brown, Chris Hans, Luis R. Pericchi, Christian P. Robert and Julyan Arbel}.
\bdoi{10.1093/acprof:oso/9780199694587.003.0017}
\bmrnumber{3204017}
\end{bincollection}
\endbibitem

\bibitem{rw06}
\begin{bbook}[author]
\bauthor{\bsnm{Rasmussen},~\bfnm{Carl~Edward}\binits{C.~E.}} \AND \bauthor{\bsnm{Williams},~\bfnm{Christopher K.~I.}\binits{C.~K.~I.}}
(\byear{2006}).
\btitle{Gaussian processes for machine learning}.
\bseries{Adaptive Computation and Machine Learning}.
\bpublisher{MIT Press, Cambridge, MA}.
\bmrnumber{2514435}
\end{bbook}
\endbibitem

\bibitem{rr24}
\begin{barticle}[author]
\bauthor{\bsnm{Ro\v{c}kov\'{a}},~\bfnm{Veronika}\binits{V.}} \AND \bauthor{\bsnm{Rousseau},~\bfnm{Judith}\binits{J.}}
(\byear{2024}).
\btitle{Ideal {B}ayesian spatial adaptation}.
\bjournal{J. Amer. Statist. Assoc.}
\bvolume{119}
\bpages{2078--2091}.
\bdoi{10.1080/01621459.2023.2241705}
\bmrnumber{4797924}
\end{barticle}
\endbibitem

\bibitem{svv13}
\begin{barticle}[author]
\bauthor{\bsnm{Szab\'{o}},~\bfnm{B.~T.}\binits{B.~T.}}, \bauthor{\bparticle{van~der} \bsnm{Vaart},~\bfnm{A.~W.}\binits{A.~W.}} \AND \bauthor{\bparticle{van} \bsnm{Zanten},~\bfnm{J.~H.}\binits{J.~H.}}
(\byear{2013}).
\btitle{Empirical {B}ayes scaling of {G}aussian priors in the white noise model}.
\bjournal{Electron. J. Stat.}
\bvolume{7}
\bpages{991--1018}.
\bdoi{10.1214/13-EJS798}
\bmrnumber{3044507}
\end{barticle}
\endbibitem

\bibitem{stan}
\begin{bmisc}[author]
\bauthor{\bsnm{{Stan Development Team}}}
(\byear{2024}).
\btitle{Stan Modelling Language Users Guide and Reference Manual v.~2.34}.
\bdoi{https://mc-stan.org}
\end{bmisc}
\endbibitem

\bibitem{vsv17}
\begin{barticle}[author]
\bauthor{\bparticle{van~der} \bsnm{Pas},~\bfnm{St\'ephanie}\binits{S.}}, \bauthor{\bsnm{Szab\'o},~\bfnm{Botond}\binits{B.}} \AND \bauthor{\bparticle{van~der} \bsnm{Vaart},~\bfnm{Aad}\binits{A.}}
(\byear{2017}).
\btitle{Uncertainty quantification for the horseshoe (with discussion)}.
\bjournal{Bayesian Anal.}
\bvolume{12}
\bpages{1221--1274}.
\bnote{With a rejoinder by the authors}.
\bdoi{10.1214/17-BA1065}
\bmrnumber{3724985}
\end{barticle}
\endbibitem

\bibitem{vkv14}
\begin{barticle}[author]
\bauthor{\bparticle{van~der} \bsnm{Pas},~\bfnm{S.~L.}\binits{S.~L.}}, \bauthor{\bsnm{Kleijn},~\bfnm{B.~J.~K.}\binits{B.~J.~K.}} \AND \bauthor{\bparticle{van~der} \bsnm{Vaart},~\bfnm{A.~W.}\binits{A.~W.}}
(\byear{2014}).
\btitle{The horseshoe estimator: posterior concentration around nearly black vectors}.
\bjournal{Electron. J. Stat.}
\bvolume{8}
\bpages{2585--2618}.
\bdoi{10.1214/14-EJS962}
\bmrnumber{3285877}
\end{barticle}
\endbibitem

\bibitem{vz08}
\begin{barticle}[author]
\bauthor{\bparticle{van~der} \bsnm{Vaart},~\bfnm{A.~W.}\binits{A.~W.}} \AND \bauthor{\bparticle{van} \bsnm{Zanten},~\bfnm{J.~H.}\binits{J.~H.}}
(\byear{2008}).
\btitle{Rates of contraction of posterior distributions based on {G}aussian process priors}.
\bjournal{Ann. Statist.}
\bvolume{36}
\bpages{1435--1463}.
\bdoi{10.1214/009053607000000613}
\bmrnumber{2418663}
\end{barticle}
\endbibitem

\bibitem{vz09}
\begin{barticle}[author]
\bauthor{\bparticle{van~der} \bsnm{Vaart},~\bfnm{A.~W.}\binits{A.~W.}} \AND \bauthor{\bparticle{van} \bsnm{Zanten},~\bfnm{J.~H.}\binits{J.~H.}}
(\byear{2009}).
\btitle{Adaptive {B}ayesian estimation using a {G}aussian random field with inverse gamma bandwidth}.
\bjournal{Ann. Statist.}
\bvolume{37}
\bpages{2655--2675}.
\bdoi{10.1214/08-AOS678}
\bmrnumber{2541442}
\end{barticle}
\endbibitem

\end{thebibliography}

\bigskip\bigskip
{\center{\large{\bf SUPPLEMENTARY MATERIAL}}}
\smallskip

\setcounter{section}{0}

\numberwithin{equation}{section}
\numberwithin{theorem}{section}
\numberwithin{lemma}{section}
\renewcommand\thesection{\Alph{section}}
\renewcommand\thefigure{\thesection.\arabic{figure}}   
\renewcommand\thetable{\thesection.\arabic{table}}   
\renewcommand\theremark{\thesection.\arabic{remark}}

{This supplement contains additional proofs and simulations. Appendix \ref{sec : simu HTlb} presents simulations regarding the lower bound results. The proof of Theorem  \ref{thm : OTconc} for the truncated Horseshoe prior distribution is presented in Appendix \ref{app:truncated}, the proof of the in-expectation lower bound of  Theorem \ref{thm : lbprob} in Appendix \ref{app:lbexp}. The proof of the key main result Theorem \ref{thm : besov sparse} providing  rates under Besov smoothness in $L_{p'}$-loss can be found in Appendix \ref{proof : besov sparse}. Finally, Appendix \ref{app:lemmas} contains a number of additional technical Lemmas.}

\section{Additional simulations: suboptimality of HT($\alpha$) priors in the undersmoothing case}\label{sec : simu HTlb}

We again consider the white noise regression model \eqref{def : GWN}, this time expanded in the orthonormal basis $\varphi_k(t)=\sqrt{2}\sin(\pi k t)$. 
leading to the normal sequence model \eqref{def : nseq}. As underlying truth, we use a function with coefficients with respect to $(\varphi_k)$ given by $f_{0,k}=k^{-2.25}\sin(10k)$. This is the setting studied in \cite[Section 3]{svv13}. In particular, this true function can be thought of as having Sobolev regularity (almost) $\beta=1.75$.

We consider HT($\alpha$) priors (given by \eqref{def : HTprior}) based on the Student distribution with $3$ degrees of freedom, for four choices of regularity $\alpha$, 0.75 and 1.25 (undersmoothing), 1.75 (match) and 2.75 (oversmoothing). Similarly to the previous subsection, we again exploit the independence structure of the model, and employ Stan to sample the one dimensional posteriors. In all cases we again truncate at $K=200$. 

In Figure \ref{fig-lb}, we present posterior sample means as well as 95\% credible regions for noise precision parameters $n=2\times 10^2$ (top row) and $n=2\times 10^4$ (bottom row). As expected by the theory, the matched and oversmoothed priors perform very well (cf. \cite[Theorem 1]{AC}), while the two undersmoothing priors lead to too rough posterior means and perform very poorly (see Theorem \ref{thm : lbprob}).

\begin{figure}
    \centering
     \includegraphics[width=0.97\textwidth]{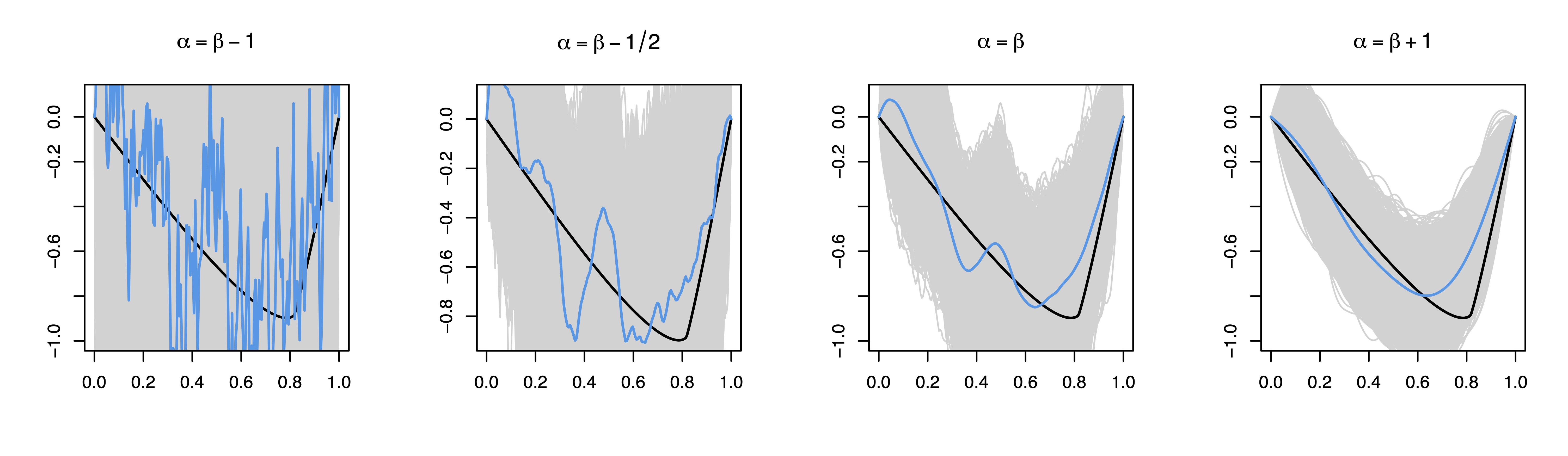}
    \includegraphics[width=0.97\textwidth]{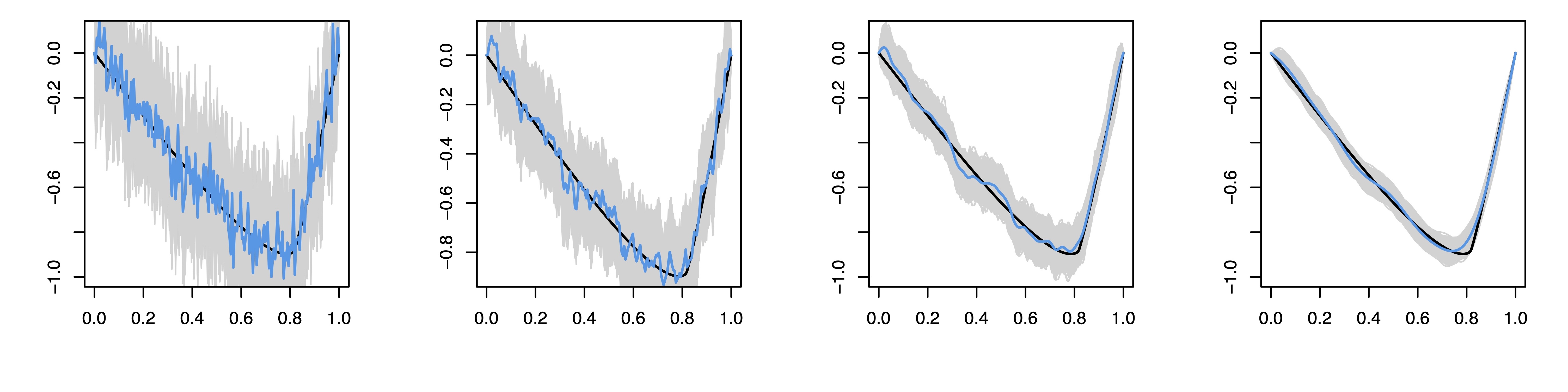}

    \caption{White noise model: true function (black), posterior means (blue), 95\% credible regions (grey), for $n=2\times 10^2$ (top row) and $n=2\times 10^4$ (bottom row). Student HT($\alpha$) prior with $\nu=3$ degrees of freedom, for $\alpha=0.75, 1.25, 1.75, 2.75$ left to right, where the Sobolev regularity of the truth $\beta$ is (almost) 1.75.}
    \label{fig-lb}
\end{figure}

\section{Proof of Theorem 1 for the truncated Horseshoe} \label{app:truncated}
Following the proof of Theorem \ref{thm : OTconc}, we take $K_n := (n/\log n)^{1/(2 \beta +1)}$ and $v_n = K_n (\log n /n) =  K_n^{-2 \beta}$. Taking care first of indices $k \leq K_n$ we only need to show,  for any $M_n \to \infty$,
\[E_{f_0} \sum_{k \leq K_n} \int (f_k-X_k)^2 \, d \Pi(f \, | \, X) = o(M_n v_n).\]
Recalling that $F$ is a bound on $f_0$, applying Lemma \ref{lem : lap} with $p=2$ and Lemma \ref{lem : lapHS} on coordinate $k$ we get, for any $t \in \R$
\[ n E_{f_0} \int (f_k-X_k)^2 \, d \Pi(f \, | \, X) \lesssim t^{-2} (1 + t^4 + \log^2 \left[ \frac{\tau \sqrt{n}}{\log \left( 1 + \frac{4 \tau^2}{( F + 1)^2} \right) }\right] ).\]
Since $\tau = n^{-a}$ with $a>0$, whenever $n$ is large enough such that $4 \tau^2 < (F+1)^2$, the right hand side of the last display is bounded (up to constant) by 
\[ t^{-2} (1 + t^4 + \log^2(\tau^{-1} \sqrt{n})).\]
This bound is optimized by taking $t$ of the order of $\sqrt{\log n}$, which leads to 
\[ E_{f_0} \sum_{k \leq K_n} \int (f_k-X_k)^2 \, d \Pi(f \, | \, X) \lesssim \sum_{k \leq K_n} \frac{\log n}{n } \lesssim \frac{K_n}{n} \log n = v_n,\]
this last bound is $o(M_n v_n)$ for any $M_n \to \infty$. 
Now for the terms $k > K_n$, since the Horseshoe prior is truncated after $k =n$, this term reduces to \[ E_{f_0} \post{  \{ f \, : \, \sum_{K_n < k \leq n} |f_k|^2 > v_n/8 \} } .\]
We can follow the rest of the proof of Theorem \ref{thm : OTconc} with $\sigma_k$ replaced by $\tau$, noting that the Horseshoe HS($1$) satisfies condition (H3) thanks to the inequality \eqref{eq : HSandwich}. Recalling $z_k := k^{-1} \log^{-2}k$, as long as $\tau^{-1}\sqrt{{Mv_n}{z_k}} \geq 1$ for $n \geq k > K_n$ (condition which is satisfied for any $\tau =n^{-a}$ and $a >1$), we have 
\[ E_{f_0} \post{  \{ f \, : \, \sum_{K_n < k \leq n} |f_k|^2 > v_n/8 \} } \lesssim \sum_{K_n < k \leq n} n^2 \frac{\tau}{\sqrt{{Mv_n}{z_k}}} + P_{f_0}(A_n^c),\]
where $A_n = \bigcap_{ K_n < k \leq n, k \in N}A_{k,0}$ is the event as in \eqref{def : eventAn}.
The sum in the last display goes to $0$ when $n \to \infty$ for $\tau = n^{-a}$ and $a \ge4$.

\section{Proof of the in-expectation lower bound in Theorem 2} \label{app:lbexp}
To show the in--expectation bound, first define the set
\[ \mathcal{N}_n = \{ k \, : \, |f_{0,k}| > 1/\sqrt{n} \}.\]
Denoting the cardinality of $\mathcal{N}_n$ as $N_n$, we have
\[ F^2 \geq \sum_{k \in \mathcal{N}_n} k^{2\beta} |f_{0,k}|^2 \geq n^{-1} \sum_{k \in \mathcal{N}_n} k^{2\beta} \geq n^{-1}\sum_{k=1}^{N_n} k^{2\beta} \geq n^{-1} (2\beta+1)^{-1}(N_n)^{2 \beta +1}. \]
Combining this with the assumption $\alpha < \beta$, we get
\[N_n \leq (nF^2(2\beta+1))^{1/(2 \beta +1)} \leq ((2\beta+1)F^2)^{1/(2 \beta +1)} K_{\alpha} =: \widetilde{K}_{\alpha}.\]
For any $k$, define the event,
\[A_k := \left\{ |X_k| \leq 2 /\sqrt{n} \right \}.\]
To establish the in--expectation lower bound, we first restrict the set of indices to
\[ \{k > \widetilde{K}_{\alpha} \} \cap \mathcal{N}_n^c = \{ k > \widetilde{K}_{\alpha} \, : \, |f_{0,k}| \leq 1/\sqrt{n}\} .\]
Using for all $k$, $|f_k - f_{0,k} |^2 \geq |f_k|^2/2 - |f_{0,k}|^2$, we get
\begin{align*}
    E_{f_0} \int \| f - f_0 \|_2^2 \, d\Pi(f \, | \, X) & \geq \sum_{\{k > \widetilde{K}_{\alpha} \} \cap \mathcal{N}_n^c}  E_{f_0}\left( \int |f_k - f_{0,k} |^2 \, d\Pi(f \, | \, X) \right) \\
    & \geq \sum_{\{k > \widetilde{K}_{\alpha} \} \cap \mathcal{N}_n^c}  E_{f_0}\left( \int \frac12 |f_k|^2 \, d\Pi(f \, | \, X) \mathbf{1}_{A_k} \right) - \sum_{\{k > \widetilde{K}_{\alpha} \} \cap \mathcal{N}_n^c} |f_{0,k}|^2.
\end{align*}
First taking care of the non-stochastic sum, since $f_0 \in S^{\beta}(F)$ and $\widetilde{K}_{\alpha} \asymp K_{\alpha}$, we have
\[\sum_{\{k > \widetilde{K}_{\alpha} \} \cap \mathcal{N}_n^c} |f_{0,k}|^2 \leq \sum_{k > \widetilde{K}_{\alpha}} |f_{0,k}|^2 \lesssim K_{\alpha}^{-2 \beta} .\]
For the remaining sum, recall that $\phi$ denotes the standard Gaussian density function, we now study
\[ E_{f_0}\left( \int |f_k|^2 \, d\Pi(f \, | \, X) \mathbf{1}_{A_k} \right) = E_{f_0}\left( \frac{\int \theta^2 \phi(\sqrt{n}(X_k - \theta)) h(\theta / \sigma_k) \, d\theta}{\int \phi(\sqrt{n}(X_k - \theta)) h(\theta / \sigma_k) \, d\theta} \mathbf{1}_{A_k} \right) =: E_{f_0}\left( \frac{T}{B} \mathbf{1}_{A_k} \right) .\]
Using $\phi \lesssim 1$, we can bound the denominator as $B \lesssim  \sigma_k$. To bound the numerator from below we set $x_k := 2 /\sqrt{n}$ such that for all $k > \widetilde{K}_\alpha$ we have $\sigma_k \leq \widetilde{K}_\alpha^{-\alpha - 1/2} \lesssim x_k$. We can then use Lemma \ref{lem : tech lower bound} with $m=2$ on the event $A_k = \{ |X_k| \leq x_k \}$, we get 
\begin{align*}
    T &\gtrsim \phi(\sqrt{n} x_k) \sigma_k^3 \gtrsim \sigma_k^3.
\end{align*}
Therefore, we obtain
\[\sum_{\{k > \widetilde{K}_{\alpha} \} \cap \mathcal{N}_n^c} E_{f_0}\left( \int \frac12 |f_k|^2 \, d\Pi(f \, | \, X) \mathbf{1}_{A_k} \right) \gtrsim \sum_{\{k > \widetilde{K}_{\alpha} \} \cap \mathcal{N}_n^c} \sigma_k^2 \, P_{f_0}(A_k).\]
For any $k \in \mathcal{N}_n^c$, we have $|f_{0,k}| \leq 1 / \sqrt{n} $, for such $k$'s, one has
\[ P_{f_0}(A_k^c) \leq Pr( |N(0,1)| > 1 ) \leq 1/2.\] 
This allows to bound the right hand side of the second to last display as \[\sum_{\{k > \widetilde{K}_{\alpha} \} \cap \mathcal{N}_n^c} \sigma_k^2 \, P_{f_0}(A_k) \gtrsim \sum_{\{k > \widetilde{K}_{\alpha} \} \cap \mathcal{N}_n^c} \sigma_k^2.\]
Recalling $ |\mathcal{N}_n| = N_n \leq \widetilde{K}_{\alpha} $ and using the fact that $(\sigma_k)$ is non-increasing, we get
\[ \sum_{\{k > \widetilde{K}_{\alpha} \} \cap \mathcal{N}_n^c} \sigma_k^2 \geq \sum_{k > 2\widetilde{K}_{\alpha}} \sigma_k^2 \gtrsim K_{\alpha}^{-2 \alpha},\]
where for the last inequality we use $\widetilde{K}_{\alpha} \asymp K_\alpha $. Finally, since $\alpha < \beta$, as $n$ is large enough,
\[ E_{f_0} \int \| f - f_0 \|_2^2 \, d\Pi(f \, | \, X)  \gtrsim K_{\alpha}^{-2 \alpha} - K_{\alpha}^{-2 \beta} \gtrsim K_{\alpha}^{-2 \alpha}.\]

\section{Proof for sparse Besov rate Theorem 3}\label{proof : besov sparse}

Throughout this proof we might write $\mathcal{L}_n$ for logarithmic factors of the form $\log^{D}n$ and some constant $D>0$, which may change from one line to another. Because we look at functions on $[0,1]$, the $L_{p'}$ spaces get smaller with $p'$. Since the minimax rate (given by \eqref{def : rateexp}) is independent of $p'$ in the regular region $\{p' \leq p\}$ we can reduce the case $p' \leq p$ to only $p'=p$. In what follows in the proof we focus on the case $p \leq p' < \infty.$ Recalling the definitions \eqref{def : indexeta} of $\eta$, \eqref{def : s'} of $s'$ and \eqref{def : rateexp} of $r$, we consider
\begin{align}\label{def : J0J1} 
2^{J_0} &: =  n^{1-2 r} \qquad \text{and} \qquad 2^{J_1} := n^{r/s'}.
\end{align}
Note that since $p' \ge p$, we have $J_0 \leq J_1$ with equality iff $p = p'$. In fact, one can see that for the homogeneous case $p=p'$, the proof becomes similar to that of Theorem \ref{thm : OTconc} (where $p=p' =2$ ) as there is a unique resolution level $J_0 = J_1$ (called $K_n$ in Theorem \ref{thm : OTconc}). To avoid repetition of the same techniques, for the rest of the proof we focus on the more elaborate inhomogeneous case $p’>p$. 
\begin{align}
    \label{eq : eq1} (p'-p)/2 + \eta(1-2r) &= rp' \quad \text{if } \eta > 0, \\
    \label{eq : eq2} (p'-p)/2 + \eta r/s' &= rp' \quad \text{if } \eta \leq 0.
\end{align}
For a function $f \in L_2$ denote by $f^{[J_1]}$ and $f^{[J_1^c]}$ the projections of $f$ onto the span of the wavelets $\{\psi_{jk}\}_{j < J_1,k}$ and $\{\psi_{jk}\}_{j \geq J_1,k}$ respectively. Consider $v_n := n^{-r} \log^\delta n$ the targeted rate, with $\delta >0$ to be chosen large enough below. Using the triangle inequality for the $L_{p'}$-norm and a union bound yields
\begin{align*}
\post{ \{ f \, : \, ||f - f_0||_{p'} > v_n \}} &\leq \post{ \{ f \, : \, ||f^{[J_1^c]} - f_0^{[J_1^c]}||_{p'} > v_n/2 \}}\\
&+ \post{ \{ f \, : \, ||f^{[J_1]} - f_0^{[J_1]}||_{p'} > v_n/2 \}}.
\end{align*}
We show that under $E_{f_0}$ both of the previous displayed terms go to $0$ with $n \to \infty$. Let us first take care of the indices $j \geq J_1$. Using Lemma \ref{lem : majlp} for $f_0$ with $J^- = J_1$ and $J^+ = \infty$, we get
\[||f_0^{[J_1^c]}||_{p'} \lesssim \sum_{j \geq J_1} 2^{j(1/2 - 1/p')} ||f_{0,j\cdot}||_{p'}.\]
Using $s-1/p =s'-1/p'$ and the embedding $\ell_{p} \subset \ell_{p'}$, available for $p' \geq p$
    \[ ||f_0^{[J_1^c]}||_{p'} \lesssim \sum_{j \geq J_1}  2^{-js'} 2^{j(1/2 +s - 1/p)} ||f_{0,j\cdot}||_{p} \lesssim \sup_{j \geq J_1}\{ 2^{j(1/2 +s - 1/p)}  ||f_{0,j\cdot}||_{p}\} \sum_{j \geq J_1}  2^{-js'}.\]
The embedding $B_{pq}^s(F) \subset B_{p\infty}^s(F)$ shows that the supremum in the last display is bounded by $F$ and 
\[ ||f_0^{[J_1^c]}||_{p'} \lesssim F 2^{-J_1 s'} \lesssim n^{-r},\]
where we have last used \eqref{def : J0J1} the definition of $J_1$.
Therefore, it holds for $n$ large enough that $||f_0^{[J_1^c]}||_{p'} \leq v_n/4$. So,
\[ \post{ \{ f \, : \, ||f^{[J_1^c]} - f_0^{[J_1^c]}||_{p'} > v_n/2 \}} \le \post{ \{ f \, : \, ||f^{[J_1^c]}||_{p'} > v_n/4 \}}.\]
Once again, applying Lemma \ref{lem : majlp} to $f^{[J_1^c]}$ with $p'$ shows that it is sufficient to control, for a large enough constant $M>0$,
\begin{equation}\label{eq : term>J1}
    E_{f_0}\post{ \{ f \, : \, \sum_{j \geq J_1} 2^{j(1/2 - 1/p')} ||f_{j\cdot}||_{p'} >M v_n \}}.
\end{equation}
Using the summability of $(2^{-j/2})_j$ (note that this sequence plays the role of ($z_k$) in the proof of Theorem \ref{thm : OTconc}), we get the following bound
\begin{align*}
    \sum_{j \geq J_1} 2^{j(1/2 - 1/p')} ||f_{j\cdot}||_{p'} = \sum_{j \geq J_1} 2^{j(1/2 - 1/p')} ( \sum_k |f_{jk}|^{p'} )^{1/p'} \le \sum_{j \geq J_1} 2^{j/2} \sup_k |f_{jk}|  \lesssim \underset{j\geq J_1 , k}{\sup}\{ 2^j |f_{jk}|\}.
\end{align*}
Plugging the previous inequality in \eqref{eq : term>J1} and applying a union bound, we are left to control, for a large enough constant $M'>0$, and any event $A_n$,
\begin{equation*}
    E_{f_0}\post{ \{ f \, : \,  \underset{j\geq J_1 , k}{\sup}\{ 2^j |f_{jk}| > M' v_n \}} \leq \sum_{j\geq J_1} \sum_k E_{f_0}\left(\post{ \{ f \, : \,  |f_{jk}| > M' 2^{-j} v_n \}} \mathbf{1}_{A_n} \right) + P_{f_0}(A_n^c).
\end{equation*}
Writing $\phi$ for the standard Gaussian density and following the proof of Theorem \ref{thm : OTconc}, we can employ Lemma \ref{lem : tech lower bound} with $m =0$ on coordinates $(j,k)$ for $x_{jk}$ to be chosen below, and obtain
\begin{equation}\label{newproof eventcontrol}
    E_{f_0}\left(\post{ \{ f \, : \,  |f_{jk}| > M' 2^{-j} v_n \}} \mathbf{1}_{A_n} \right) \lesssim E_{f_0} \left( \frac{1}{\phi(\sqrt{n}x_{jk})} 2 \overline{H}\left( \sigma_j^{-1}M'v_n 2^{-j}\right) \mathbf{1}_{A_n} \right) .
\end{equation}
Now for any $j \geq J_1 \gtrsim \log_2 n$, since $\sigma_j = 2^{-j^2}$, for $n$ large enough we have $\sigma_j^{-1}M'v_n 2^{-j} \geq 1$ so we use assumption (H3) and obtain
\[\overline{H}\left( \sigma_j^{-1}M'v_n 2^{-j}\right) \lesssim \frac{2^j\sigma_j}{M' v_n}.\]
To define the event $A_n$, first consider the sets of indices (the set $\Lambda_n$ will be usefull further down the proof for indices $j < J_1$)
\begin{equation}\label{eq : newdeflambdaI}
     \Lambda_n := \left\{ (j,k) \, : \, J_0 < j < J_1, \, |f_{0,jk}| \leq 1/ \sqrt{n} \right\} \qquad \text{and} \qquad \mathcal{I}:= \Lambda_n \cup \{(j,k) \, : \, j \geq J_1 \}.
\end{equation}
Now consider for any $l \geq 0$
\[ \mathcal{A}_{jk,l} := \left\{ |X_{jk}| \leq \sqrt{\frac{4(l+1) \log n}{n}}\right\},\]
and define the event
\begin{equation}\label{def : newAn}
    A_n := \bigcap_{J_0 < j \leq \log_2 n \, , \, (j,k) \in \mathcal{I}} \mathcal{A}_{jk,0} \, \cap \,\bigcap_{l \geq 1}  \bigcap_{l \log_2 n < j \leq (l+1) \log_2 n \, , \, (j,k) \in \mathcal{I}} \mathcal{A}_{jk,l}.
\end{equation}
For $l\geq 0$, let us set $x_{jk} := \sqrt{(4(l+1)\log n)/n}$ whenever $l \log_2(n) <j \leq (l+1) \log_2 n $. For every $j > J_0$ we have $\sigma_j \lesssim 1/\sqrt{n}$ when $n$ is large enough, therefore the constrains $|X_{jk}|\leq x_{jk}$ and $\sigma_j \lesssim x_{jk}$ are satisfied on the event $A_n$. We can use Lemma \ref{lem : tech lower bound} with $m=0$ on the event $A_n$ to obtain the bound in \eqref{newproof eventcontrol}. Noting that $\phi(\sqrt{n}x_{jk}) \gtrsim n^{-2(l+1)}$ when $l \log_2(n) <j \leq (l+1) \log_2 n $ leads to 
\begin{align*}
     & \sum_{j\geq J_1} \sum_kE_{f_0}\left(\post{ \{ f \, : \,  |f_{jk}| > M' 2^{-j} v_n \}} \mathbf{1}_{A_n} \right) \\
     & \lesssim  \sum_{J_0 < j \leq \log_2 n} \sum_k n^2 \frac{2^j\sigma_j}{M' v_n} + \sum_{l \geq 1}\sum_{j = l \log_2 n - 1}^{(l+1) \log_2 n} \sum_k n^{2(l+1)} \frac{2^j\sigma_j}{M' v_n}.
\end{align*}
Recalling $\sigma_j = 2^{-j^2}$, $v_n = n^{-r} \log^\delta n$ and $J_0 \asymp \log_2n$ the sums of the last display satisfy, as $n \to \infty$
\[\sum_{J_0 < j \leq \log_2 n} \sum_k n^2 \frac{2^j\sigma_j}{M' v_n}=\sum_{J_0 < j \leq \log_2 n} n^2\frac{2^{2j}2^{-j^2}}{M' v_n} \leq \frac{n^4}{M' v_n} 2^{-J_0^2}\log_2n = o(1),\]
\begin{align*} \sum_{l \geq 1}\sum_{j = l \log_2 n - 1}^{(l+1) \log_2 n} \sum_k n^{2(l+1)}& \frac{2^j\sigma_j}{M' v_n}=\sum_{l \geq 1}\sum_{j = l \log_2 n - 1}^{(l+1) \log_2 n} n^{2(l+1)} \frac{2^{2j}2^{-j^2}}{M' v_n} \\&\leq  \frac{1}{M' v_n}\sum_{l \geq 1}  n^{4(l+1)} 2^{-l^2 \log_2^2n} = o(1).\end{align*}
Using Lemma \ref{lem : pAn} gives $P_{f_0}(A_n^c) \to 0$, as $n \to \infty$, we obtain 
\[E_{f_0}\post{ \{ f \, : \, ||f^{[J_1^c]} - f_0^{[J_1^c]}||_{p'} > v_n/2 \}} = o(1).\]
We are now left to control the indices $j < J_1$. Applying Lemma \ref{lem : majlp} with the $L_{p'}$ norm, $J^- = -1$ and $J^+ = J_1$ and noting that $J_1 \lesssim \log n$, leads to
\begin{align*}
    || f^{[J_1]} - f_0^{[J_1]} ||_{p'} &\lesssim (\log n)^{1 - 1/p'} \left( \sum_{j < J_1} \sum_k 2^{j(p'/2 - 1)}|f_{jk} - f_{0,jk}|^{p'}\right)^{1/p'}.
\end{align*}
If $\delta>0$ is large enough, it is then sufficient to control
\begin{equation}\label{eq : term<J1}
    E_{f_0} \post{ \{ f \, : \, \sum_{j < J_1} \sum_k 2^{j(p'/2 - 1)}|f_{jk} - f_{0,jk}|^{p'} > Mv_n^{p'} \}}.
\end{equation}
Let $\mathcal{N}_n$ be the set of indices $(j,k)$ with resolution level $j \in (J_0,J_1)$ and high valued signal coefficients
\begin{equation}\label{def : highval}
    \mathcal{N}_n := \left\{ (j,k) \, : \, J_0<j<J_1, \, |f_{0,jk}| > 1/ \sqrt{n} \right\}.
\end{equation}

Recall \eqref{eq : newdeflambdaI} the definition of the set $\Lambda_n$, we have the decomposition
\[\{ (j,k) \, : \, j < J_1\}=\{ (j,k) \, : \, j \leq J_0\} \cup \mathcal{N}_n \cup \Lambda_n.\]
Using the triangle inequality and a union bound we split the sum in \eqref{eq : term<J1} into three terms according to the previous decomposition. The first of these three terms is  
\begin{equation*}
    E_{f_0} \post{ \{ f \, : \, \sum_{j < J_0} \sum_k 2^{j(p'/2 - 1)}|f_{jk} - f_{0,jk}|^{p'} > (M/3)v_n^{p'} \}} .
\end{equation*} 
Applying Markov's inequality reduces the study of this first term to obtain a bound on
\begin{equation}\label{eq : defA}
    \mathbf{A} :=\sum_{j \leq J_0} \sum_k 2^{j(p'/2 - 1)} E_{f_0}  \int |f_{jk} - f_{0,jk}|^{p'} \, d \Pi(f \, | \, X)  .
\end{equation} 
From the convexity inequality $|x+y|^{p'} \lesssim |x|^{p'} + |y|^{p'}$, one gets
\[ |f_{jk} - f_{0,jk}|^{p'} \lesssim |f_{jk} - X_{jk}|^{p'} + |X_{jk} - f_{0,jk}|^{p'}.\]
Under $E_{f_0}$, the observation $(X_{jk})$ in model \eqref{def : nseqwav} is distributed as $\mathcal{N}(f_{0,jk},1/n)$, therefore $E_{f_0}(|X_{jk}-f_{0,jk}|^{p'}) \lesssim n^{-p'/2}$. Along with $\mathbf{1}_{A_n}\leq 1$, this leads to the bound,
\begin{equation}\label{eq : boundA}
    \mathbf{A} \lesssim n^{-p'/2} \sum_{j \leq J_0} 2^{jp'/2} + \sum_{j \leq J_0} \sum_k 2^{j(p'/2 - 1)} E_{f_0} \left(\int |f_{jk} - X_{jk}|^{p'} \, d \Pi(f \, | \, X)\right).
\end{equation} 
Applying Lemma \ref{lem : lap} with $p'$ on coordinates $\{(j,k) \, : j \leq J_0$\}, gives, for any $t > 0$
\[E_{f_0} \left( \int |f_{jk} - X_{jk}|^{p'} \, d \Pi(f \, | \, X) \right) \lesssim (t\sqrt{n})^{-p'} \left[ 1 + \log^{p'} E_{f_0} \left( \int e^{t\sqrt{n}|X_{jk}-f_{0,jk}|} \, d\Pi(f \, | \, X)\right) \right] .\]
Note that since $f_0 \in B_{pq}^s(F)$, we have $|f_{0,jk}| \leq F$, using Lemma \ref{lem : lapHTnew} to bound the Laplace transform of the posterior with prior $f_{jk} = \sigma_j \zeta_{jk}$ and using $|x+y|^{p'} \lesssim |x|^{p'} + |y|^{p'}$,

\[E_{f_0}\left( \int |f_{jk} - X_{jk}|^{p'} \, d \Pi(f \, | \, X) \right) \lesssim (\sqrt{n}t)^{-p'}\left[ 1 + \log^{p'}(\sigma_j \sqrt{n}) +  t^{2p'} + \log^{p'(1 + \kappa)}\left( 1 + \frac{F + 1/\sqrt{n}}{\sigma_j} \right)\right].\]
Now since $\sigma_j = 2^{-j^2}$ and $j \leq J_0 \asymp \log n$ we have $\sigma_j^{-1} \leq e^{C \log^2 n }$ for some $C > 0$. The bound then becomes, as $n$ gets large enough,

\[E_{f_0}\left( \int |f_{jk} - X_{jk}|^{p'} \, d \Pi(f \, | \, X) \right) \lesssim (\sqrt{n}t)^{-p'}\left( 1 +  t^{2p'} + \log^{2p'(1 + \kappa)}n \right).\]
This bound is optimized by $t^{2p'}$ of the order of the log factor in the last display which leads to
\begin{equation}\label{eq : boundafterlaplace}
    E_{f_0}\left( \int |f_{jk} - X_{jk}|^{p'} \, d \Pi(f \, | \, X) \right) \lesssim n^{-p'/2}\mathcal{L}_n.
\end{equation}
Recall \eqref{def : J0J1} the definition of $J_0$, plugging the previous bound in \eqref{eq : boundA} leads to 
\[ \mathbf{A} \lesssim \mathcal{L}_n n^{-p'/2} \sum_{j \leq J_0} 2^{jp'/2} \lesssim \mathcal{L}_n \left( \frac{2^{J_0}}{n}\right)^{p'/2} \lesssim \mathcal{L}_n n^{-rp'} .\]
This last bound shows that $\mathbf{A}=O(v_n^{p'})$ for $\delta >0$ large enough, ensuring that, as $n \to \infty$,
\[E_{f_0} \post{ \{ f \, : \, \sum_{j \leq J_0} \sum_k 2^{j(p'/2 - 1)}|f_{jk} - f_{0,jk}|^{p'} > (M/3)v_n^{p'} \}} = o(1)\]
and thus taking care of the first term in the decomposition of \eqref{eq : term<J1}. The second term we need to bound is the sum restricted on $\mathcal{N}_n$ (defined in \eqref{def : highval}). Applying again Markov's inequality leave us to bound
\begin{equation}\label{eq : defB}
    \mathbf{B} := \sum_{(j,k) \in \mathcal{N}_n} 2^{j(p'/2 - 1)} E_{f_0} \int |f_{jk} - f_{0,jk}|^{p'} \, d \Pi(f \, | \, X) .
\end{equation}
On the set $\mathcal{N}_n \subset \{(j,k) \, : \, J_0 < j < J_1\}$, using the monotonicity of $(\sigma_j)$ and recalling the definition \eqref{def : J0J1} of $J_1$, we obtain $\log(\sigma_j^{-1}) \leq \log(\sigma_{J_1}^{-1}) \leq \log(2^{J_1^2}) \lesssim \log^2 n$. We now do the same split as in the term $\mathbf{A}$ and use Lemmas \ref{lem : lap} and \ref{lem : lapHTnew} to obtain the bound \eqref{eq : boundafterlaplace} in the term $\mathbf{B}$. Moreover, since $|f_{0,jk}| > 1/\sqrt{n}$ on $\mathcal{N}_n$ we have 
\begin{align*}
    \mathbf{B} &\lesssim \mathcal{L}_n n^{-p'/2} \sum_{(j,k) \in \mathcal{N}_n} 2^{j(p'/2-1)} \lesssim \mathcal{L}_n  n^{(p-p')/2} \sum_{(j,k) \in \mathcal{N}_n} 2^{j(p'/2-1)} |f_{0,jk}|^p.
\end{align*}
 Using first $\mathcal{N}_n \subset \{(j,k) \, : \, J_0 <j < J_1\}$ and then $||f_{0,j\cdot}||_p^p\leq F^p 2^{-jp(s -1/p +1/2)}$, we get
\begin{align*}
    \mathbf{B} \lesssim \mathcal{L}_n n^{(p-p')/2} \sum_{J_0 < j < J_1}2^{j(p'/2-1)} ||f_{0,j\cdot}||_p^p \lesssim \mathcal{L}_n n^{(p-p')/2} \sum_{J_0 < j < J_1} 2^{-j \eta},
\end{align*}
where $\eta$ is the index defined in \eqref{def : indexeta}. If $\eta < 0$, recalling \eqref{def : J0J1} the definition of $J_1$, we employ equation \eqref{eq : eq2} and obtain
\[ \mathbf{B} \lesssim \mathcal{L}_n n^{(p-p')/2} 2^{-J_1\eta} \lesssim \mathcal{L}_n n^{(p-p')/2} n^{- \eta r/s'} \lesssim \mathcal{L}_n n^{- r p'}.\]
If $\eta > 0$, applying equation \eqref{eq : eq1} we find
\[ \mathbf{B} \lesssim \mathcal{L}_n n^{(p-p')/2} 2^{-J_0 \eta} \lesssim \mathcal{L}_n n^{-r p'}.\]
Finally, when $\eta = 0$, we have $p' = (2 s +1)p$ and therefore $(p'-p)/2 = sp = r p' $ 
such that $\mathbf{B} \lesssim \mathcal{L}_n n^{-r p'}$ (note that in the case $\eta = 0$ there is an extra log factor in $\mathcal{L}_n$). We have shown that $\mathbf{B} = O(v_n^{p'})$ and thus the second term in the decomposition of \eqref{eq : term<J1} satisfies, as $n \to \infty$,
\[E_{f_0} \post{ \{ f \, : \, \sum_{(j,k) \in \mathcal{N}_n} 2^{j(p'/2 - 1)}|f_{jk} - f_{0,jk}|^{p'} > (M/3)v_n^{p'} \}} = o(1).\]
The last term we need to control is the probability involving the sum restricted to $\Lambda_n$. Recalling that $|f_{0,jk}| \leq 1/\sqrt{n}$ on $\Lambda_n$ and using $p' \geq p$, we can see that
\[ \sum_{(j,k) \in \Lambda_n} 2^{j(p'/2-1)} |f_{0,jk}|^{p'} \leq n^{(p-p')/2} \sum_{J_0 < j < J_1} \sum_k 2^{j(p'/2 -1)} |f_{0,jk}|^p ,\]
which is bounded as in the study of the $\mathbf{B}$ term by $ \mathcal{L}_nn^{- r p'} \leq v_n^{p'}$ for $\delta >0$ large enough. Using $|f_{jk} - f_{0,jk}|^{p'} \lesssim |f_{jk}|^{p'} + |f_{0,jk}|^{p'}$, it is then sufficient to control
\[E_{f_0} \post{ \{ f \, : \, \sum_{(j,k) \in \Lambda_n} 2^{j(p'/2 - 1)}|f_{jk}|^{p'} > (M/3)v_n^{p'} \}} .\]
A union bound can be used as in the previous study of the indices $j > J_1$, this leaves us to control
\[\sum_{(j,k) \in \Lambda_n}E_{f_0} \post{ \{ f \, : \,  |f_{jk}| > M'2^{-j}v_n \}} .\]
Noting that $\Lambda_n \subset \{ (j,k) \, : \, J_0 < j <J_1 \}$ and that $J_0 \asymp \log n$, we follow the same steps as in the aforementioned study and show that the quantity in the last display goes to $0$ with $n \to \infty$, achieving the control of the three terms and ensuring that, as $n \to \infty$,
\[E_{f_0}\post{ \{ f \, : \, ||f^{[J_1]} - f_0^{[J_1]}||_{p'} > v_n/2 \}} = o(1),\]
which concludes the proof.

\section{Additional Lemmas} \label{app:lemmas}

\begin{lemma}\label{lem : lapHS}
    Let $\theta \sim HS(\tau)$, for $\tau >0$ and $X|\theta \sim \mathcal{N}(\theta, 1/n)$. Then for some $C > 0$, it holds, for all $t \in \R$, $\theta_0 \in \R$, $\tau> 0$,

\[ E_{\theta_0} E_{ \theta \sim HS(\tau)} \left[e^{t\sqrt{n}(\theta-X)} \, | \, X \right] \leq C \frac{\tau \sqrt{n}}{\log \left( 1 + \frac{4 \tau^2}{( |\theta_0| + 1)^2} \right)} e^{t^2/2}\]
\end{lemma}

\begin{proof}
    One can follow the proof of Lemma \ref{lem : lapHTnew} and employ the sandwich inequality \eqref{eq : HSandwich} for the Horseshoe to bound the denominator instead of the general heavy-tail lower bound (H2).
\end{proof}

\begin{lemma}[Bounding $L_p$--norms in terms of wavelet coefficients]\label{lem : majlp}
    Let  $J^+ \geq J^- \geq -1$. Denote $\mathcal{J} := \{j\, : \, J^- \leq j \leq J^+ \}$ and $\mathcal{K}_j :=\{ k \, : \, 0 \leq k <2^j \} $ for any $j \geq -1$. Consider \[g := \sum_{j \in \mathcal{J}} \sum_{k \in \mathcal{K}_j} g_{jk}\psi_{jk}. \] For any $p \geq 1$, we have
    \begin{align*}
        ||g||_p  \lesssim \sum_{j \in \mathcal{J}} 2^{j(1/2 - 1/p)} ||g_{j \cdot}||_p \qquad \text{and} \qquad
        ||g||_p^p \lesssim (J^+ - J^-)^{p-1} \sum_{j\in \mathcal{J}}2^{j(p/2 - 1)} ||g_{j \cdot}||_p^p
    \end{align*}
    
\end{lemma}
\begin{proof}
   Using the triangle inequality and the Parseval-like equality \eqref{parseval}, we have
   \[ ||g||_p =  \lVert \sum_{j \in \mathcal{J}} \sum_{k \in \mathcal{K}_j} g_{jk} \psi_{jk} \rVert_p 
             \leq \sum_{j \in \mathcal{J}}  \lVert  \sum_{k \in \mathcal{K}_j} g_{jk} \psi_{jk}  \rVert_p 
             \lesssim \sum_{j \in \mathcal{J}} 2^{j(1/2 - 1/p)} ||g_{j \cdot}||_p .\]
    Now simply apply Hölder's inequality with exponents $p$ and $p/(p-1)$ to get the desired result.
\end{proof}

\begin{lemma}\label{lem : pAnk}
    Let $A_n$ be the event defined in \eqref{def : newAn} and assume $f_0 \in B_{pq}^s(F)$, with $p,q \geq 1 $ and $ s,F>0$. Then as $n \to \infty$, we have
    \[ P_{f_0}(A_n^c) \to 0.\]
\end{lemma}
\begin{proof}
 We notice that for any $(j,k) \in \mathcal{I}$, by definition, $|f_{0,jk}| \leq 1/\sqrt{n}$. Therefore the proof is similar to the single-index version of this lemma, Lemma \ref{lem : pAn}, with $l$ taking the role of $\log \ell$ therein, we have in particular, as $n$ gets large enough, for all $(j,k) \in \mathcal{I}$ and $l \geq 0$
    \[P_{f_0}(\mathcal{A}_{l,jk}^c) \leq P\left\{ \left| \mathcal{N}(0,1) \right| >\sqrt{3 (l+1) \log n }\right\} \leq n^{- \frac32(l+1)}.\]
The last bound along with a union bound yield
    \[ P_{f_0}(A_n^c)\lesssim \sum_{l \geq 0} 2^{(l+1) \log_2 n }n^{- \frac32(l+1)} \lesssim n^{-1/2} \sum_{l \geq 0}n^{-l/2} \lesssim 1 /\sqrt{n},\]
   ensuring $P_{f_0}(A_n^c) \to 0$ as $n \to \infty$.
\end{proof}

\end{document}